\RequirePackage{fix-cm}
\documentclass[smallextended]{svjour3}       
\smartqed  
\usepackage{graphicx}
\usepackage{amsfonts}
\usepackage{graphicx}
\usepackage{epstopdf}
\usepackage{algorithmic}
\usepackage{subfigure}
\usepackage{color}

\usepackage{amsopn}

\usepackage{bm}
\usepackage{amsmath}
\usepackage{placeins}

\newcommand{\dom}{dom}
\newcommand{\rank}{rank}

\newcommand{\img}{img}

\spnewtheorem{assumption}{Assumption}{\bf}{\it}
\spnewtheorem{algorithm}{Algorithm}{\bf}{\it}

\newcommand{\T}{\mathsf{T}}

\begin{document}

\title{A Convergent ADMM Algorithm for Grain Boundary Energy Minimization}


\author{Yue Wu, Luchan Zhang, Yang Xiang }


\institute{Y. Wu \at
          Department of Mathematics,\\
          The Hong Kong University of Science and Technology, Clear Water Bay, Kowloon, Hong Kong, China.  \\
          \email{ywudb@connect.ust.hk}
          \and L. C. Zhang \at School of Mathematical Sciences, Shenzhen University, Shenzhen, 518060,  China. \\
          Corresponding author.
          \email{zhanglc@szu.edu.cn}
          \and Y. Xiang \at Department of Mathematics, Hong Kong University of Science and Technology, Clear Water Bay, Kowloon, Hong Kong, China.
          \at HKUST Shenzhen-Hong Kong Collaborative Innovation Research Institute,  Futian, Shenzhen, China.\\
          Corresponding author.
          \email{maxiang@ust.hk}
}
\date{Received: date / Accepted: date}

\maketitle

\begin{abstract}

In this paper, we study a constrained minimization problem that arises from materials science to determine the dislocation (line defect) structure of grain boundaries. The problem aims to minimize the energy
of the grain boundary with dislocation structure subject to the constraint of Frank’s formula.
In this constrained minimization problem, the objective function, i.e., the grain boundary energy, is
nonconvex and separable, and the constraints are linear.
To solve this constrained minimization problem, we modify the alternating direction method of multipliers (ADMM) with an increasing penalty parameter.
We provide a
convergence analysis of the modified ADMM in this nonconvex minimization
problem, with  settings not considered by the existing ADMM convergence studies. Specifically, in the linear constraints, the coefficient matrix of each subvariable block is of full column rank. This property makes each subvariable
minimization strongly convex if the penalty parameter is large enough, and
contributes to the convergence of ADMM without any convex assumption on
the entire objective function. We prove that the limit of the sequence from the
modified ADMM is primal feasible and is the stationary point of the augmented
Lagrangian function. Furthermore, we obtain sufficient conditions to show
that the objective function is quasi-convex and thus it has a unique minimum
over the given domain. Numerical examples are presented to validate the convergence of the algorithm, and
results of the penalty method, the augmented Lagrangian method, and the modified ADMM are compared.

\end{abstract}

\section{Introduction}
Grain boundary energy and grain boundary motion in crystalline materials, which strongly depend on the microstructure of grain boundaries, play important roles in the materials properties \cite{table1,table2,ReadShockley1950,Frank1950,Cahn2004,Olmsted2009,Lim2009,Wu2012,Zhu2014,Dai2014,Bulatov2014,ZHANG,Zhang2018,Voigt2018,Liu2019gb,Qin-static2021,Qin2021}.
For low angle grain boundaries, their energetic and dynamic properties depend  on the structure of dislocations (line defects) which satisfies the Frank's formula \cite{table1,table2,ReadShockley1950,Frank1950}.
However, the Frank's formula is not able to uniquely determine the dislocation structure of a low angle grain boundary.
In Ref.~\cite{ZHANG}, a continuum model was proposed to compute the structure and energy of a low angle grain boundary given the boundary plane orientation.
The idea of this continuum model is to find the grain boundary structure that  minimizes the grain boundary energy subject to the constraint of the Frank’s formula.
This constrained minimization problem was solved by the penalty method in Ref.~\cite{ZHANG}. The continuum model has been generalized to dislocation structure and energy of curved grain boundaries and solved using augmented Lagrangian method \cite{Qin-static2021}, and to interfaces between different materials that have lattice mismatch \cite{Zhang-cms}. This continuum model for energy and dislocation structure of grain boundaries and interfaces and the associated results have been successfully applied to different materials systems \cite{Srolovitz2018,Wan-Acta2023,Du2023,Furstoss2024,Lv2024,Zhang2024}.
In such a constrained minimization problem, the objective energy function  is nonconvex,
which makes analysis of existence and uniqueness of the solution challenging. Moreover, convergence of the algorithms based on the penalty method or augmented Lagrangian method has not been examined.

%

In this paper, 
we propose an efficient numerical method to solve this constrained energy minimization problem based on modification of the alternating direction method of multipliers (ADMM) algorithm \cite{Distributed}.
Here the objective function is nonconvex over a high dimensional space, which makes it challenging to solve the constraint minimization problem efficiently.
The multi-block separable form of the objective function enables the adaptation of the ADMM algorithm \cite{Distributed},
which is an efficient method to solve such minimization problems that decomposes a one-step minimization in a high-dimensional space  into multiple minimization steps with each in a lower-dimensional subspace.
We modify the standard ADMM algorithm by using
  an increasing penalty parameter, and prove the convergence of the proposed algorithm for solving our constrained energy minimization problem.

The convergence of ADMM for this constrained minimization problem has several challenges.
Firstly, it has been pointed out that direct extension of ADMM to multi-block problems may not necessarily converge \cite{Caihua}, although convergence of two-block ADMM has been proved \cite{Distributed}.
Secondly, many analyses on the convergence of the multi-block ADMM are under certain convexity assumptions on the objective function \cite{Caihua,Yuan,chen2025convergence}.
Thirdly, although there are results on the convergence of ADMM for some special nonconvex models  \cite{Hong,Li,Yin,hien2024multiblock,yashtini2021multi},
these specific forms do not apply to this constrained minimization problem being considered for grain boundary dislocation structure.

We give a proof of the convergence of the modified ADMM algorithm for this constrained minimization problem for grain boundary dislocation structure,
and the limit is a stationary point of this constrained minimization problem.
We identify a property of the constraints that helps the convergence proof of ADMM.
Specifically, in our problem, the constraints are linear and the coefficient matrix associated with the constraints consists of $J$ Burgers vectors for the dislocations which are not zero vectors.
As a result,  even though the entire coefficient matrix is not of full rank,
the coefficient matrix $A_j \in \mathbb{R}^{6\times2}$ of each variable block $ \bm{u_j}, j=1,\dots,J$, has full column rank and each matrix $\frac{1}{b} A_j$ is semi-orthogonal, which leads to $ \| A_{j} (y - x) \|^2 = b^2 \| y - x\|^2, \forall x,y \in \mathbb{R}^2$.
This property makes the minimization of each subvariable block strongly convex. Our proof also provides a new type of nonconvex problems for which ADMM converges.

Moreover,
we employ sufficient conditions of quasi-convexity to prove that the stationary point of the objective function is unique in a bounded closed domain.
 Under this uniqueness result, the solution of the minimization problem is unique,
and ADMM converges to the optimum of this constrained minimization problem.
The uniqueness of the minimum also guarantees that a
minimum obtained by numerical methods such as the penalty method and the augmented Lagrangian method (ALM)
is the solution of the minimization problem.


This paper is organized as follows. In Section~\ref{sec:problem}, we present the constrained minimization problem in Ref.~\cite{ZHANG} for the dislocation structure of grain boundaries.
In Section~\ref{sec:ADMM algorithm}, we introduce the modified ADMM algorithm with an increasing penalty parameter and analyze the convergence of this modified ADMM algorithm; we  also compare the properties of the constrained minimization problem for grain boundary dislocation structure with assumptions in  the available convergence proofs for the ADMM algorithm for minimization problems with nonconvex objective functions.
In Section~\ref{sec:unique}, we introduce sufficient conditions for quasi-convexity functions and employ them to show that
the stationary point of the objective function in our minimization problem is unique in a bounded closed domain;
we also show the minimum of the constrained minimization problem over a bounded and closed domain is unique, in three Burgers vectors case.
 In Section~\ref{sec: numerical}, we present simulation results of using our modified ADMM algorithm to solve the constrained minimization problem, and the results are compared with those using the penalty method and the ALM.
Conclusions are presented in Section~\ref{sec:conclusions}.

\section{Constrained minimization problem}
\label{sec:problem}
The constrained minimization problem for dislocation structure of a low angle grain boundary proposed in Ref.~\cite{ZHANG} is:
\begin{equation} \label{eq:of ori}
	\min \gamma_{g b}
\end{equation}
where
\begin{equation}\label{eq:gamma-gb}
\gamma_{g b}= \sum_{j=1}^J \frac{\mu (b^{(j)})^2}{4\pi (1-\nu)}
	\left[1-\nu \frac{(\nabla \eta_j \times \bm{n} \cdot \bm{b}^{(j)})^2 }{(b^{(j)})^2 \|\nabla \eta_j \|^2}\right]
	\| \nabla \eta_j \| \log \frac{1}{r_g \sqrt{\| \nabla \eta_j\|^2 +\epsilon  }},
\end{equation}
subject to
\begin{equation} \label{eq:frank}
	\theta(\bm{V} \times \bm{a} ) - \sum_{j=1}^J \bm{b}^{(j)} (\nabla \eta_j \cdot \bm{V}) = \bm{0}.
\end{equation}

Here the objective function $\gamma_{gb}$ represents the energy density of this planar low angle grain boundary, which depends on the dislocation structure on the grain boundary. On the grain boundary, there are dislocations with
 $J $ different Burgers vectors $\bm{b}^{(j)}= (b_{j1},b_{j2}, b_{j3})^\T, j=1,2,\dots,J $, with length $b^{(j)}$. Here the superscript $\T$ represents transpose. The arrangement of the dislocation array with Burgers vector $\bm{b}^{(j)}$ is represented by the dislocation density potential function $\eta_j$ \cite{Xiang1,Zhu2014,ZHANG} defined on the grain boundary plane, whose
 gradient $ \nabla \eta_j$ gives the normal direction of the dislocations in the grain boundary and $1/\| \nabla \eta_j\|$ gives the inter-dislocation distance of this dislocation array.
  In the objective function Eq.~\eqref{eq:gamma-gb}, $\bm{n}$ is the unit normal vector of the grain boundary plane,
$\mu$ and $\nu$ are elastic constants,  $r_g$ is a dislocation core parameter, and
$\epsilon $ is a small positive regularization constant to avoid the numerical singularity when $\| \nabla \eta_j \| =0$.

%

The linear constraints are based on the fact that the dislocation structure of an equilibrium planar low angle grain boundary should satisfy the Frank's formula.
In the constraint Eq.~\eqref{eq:frank},
$ \theta$ is a constant parameter, representing the misorientation angle of the grain boundary,
$\bm{V} $ is any vector in the grain boundary,
and $\bm{a}=(a_1,a_2,a_3)^\T$ is the unit vector along the rotation axis of the grain boundary.

In this section, we consider the constrained minimization problem in a simplified setting where
the Burgers vectors $\bm{b}^{(j)}, j = 1,2,\dots, J$ are  of the same length $b$, where $b={b }^{(j)}=\sqrt{ b_{j1}^2+b_{j2}^2+b_{j3}^2}$. This assumption holds for dislocations in fcc (face-centered cubic) crystals \cite{table2}.
We use $1/b$ as the unit of $\nabla \eta_j$.
Using $\frac{\mu b}{4\pi (1-\nu)}$ as the unit of  grain boundary energy density, 
the objective function in the dimensionless form becomes
\begin{equation} \label{eq:of}
	 \gamma_{g b}= \sum_{j=1}^J
	\left[1-\nu \frac{(\nabla \eta_j \times \bm{n} \cdot \bm{b}^{(j)})^2 }{ b^2 (\|\nabla \eta_j \|^2+\epsilon)}\right]
	b\sqrt{\| \nabla \eta_j \|^2 +\epsilon  }    \log \frac{1}{r_g \sqrt{\| \nabla \eta_j\|^2 +\epsilon  }}.
\end{equation}
We denote
$$ \bm{u}_j=\nabla \eta_j =(u_{jx} ,u_{jy})^\T \in \mathbb{R}^2,j=1,\dots,J, \quad
\bm{u} = (\bm{u}_1, \dots, \bm{u}_J)^\T \in \mathbb{R}^{2J},  $$
and define functions
$f_j: \mathbb{R}^2 \longrightarrow \mathbb{R},$
\begin{equation} \label{eq:f}
	f_j(\bm{u}_j) = 	\left[1-\nu \frac{(\bm{u}_j \times \bm{n} \cdot \bm{b}^{(j)})^2 }{(b^{(j)})^2 (\|\bm{u}_j \|^2+\epsilon)}\right]
	b\sqrt{\|\bm{u}_j \|^2 +\epsilon  }    \log \frac{1}{r_g \sqrt{\|\bm{u}_j\|^2 +\epsilon  }} .
\end{equation}
Note that each component function $f_j(\bm{u}_j)$ is nonconvex; see Figure~\ref{fig:fsc} in Sec.~\ref{sec:4.2} for an example,
and the objective function
\begin{equation}
    \gamma_{gb}(\bm{u}) = \sum_{j =1}^{J} f_j( \bm{u}_j)
\end{equation}
is separable.
The linear constraint Eq.~\ref{eq:frank} holds if and only if it holds for two basis vectors in the $xy$ plane: $ \bm{V}=\bm{V}_1=(1,0,0)^\T$ and $\bm{V}=\bm{V}_2=(0,1,0)^\T.$
Hence we can rewrite the constraints as six equations:
\begin{equation}\label{eq:h}
	\begin{aligned}
		&-\sum_{j=1}^J b_{j1} u_{jx} = 0, &
		-\sum_{j=1}^J b_{j2} u_{jx} - \theta a_3=0, \\
		&-\sum_{j=1}^J b_{j3} u_{jx} + \theta a_2=0, &
		-\sum_{j=1}^J b_{j1} u_{jy} + \theta a_3=0, \\
		&-\sum_{j=1}^J b_{j2} u_{jy} = 0, &
		-\sum_{j=1}^J b_{j3} u_{jy} - \theta a_1= 0.
	\end{aligned}
\end{equation}
Denote the coefficient matrix of each variable block $ \bm{u}_j$ in the linear system Eq.~\eqref{eq:h} by
\begin{equation}\label{eq:Aj}
	A_j =
	\left[
	\begin{array}{cc}
		\bm{b}^{(j)} & \bm{0} \\
		\bm{0}&  \bm{b}^{(j)}
	\end{array}
	\right]
	=\left[\begin{array}{cc}
		b_{j1} & 0 \\
		b_{j2} & 0 \\
		b_{j3} & 0 \\
		0 & b_{j1} \\
		0 & b_{j2} \\
		0 & b_{j3}
	\end{array}\right]  \in \mathbb{R}^{6 \times 2},
\end{equation}
and
\begin{equation}
\bm{c} =(0 , -\theta a_3, \theta a_2,\theta a_3,0, -\theta a_1 )^\T \in \mathbb{R}^{6 }.
\end{equation}

Then the optimization problem becomes:
\begin{equation}\label{eq:cp}
	\begin{aligned}
		\min  &  \sum_{j =1}^{J} f_j( \bm{u}_j)  \\
		\text{s.t. }&  \sum_{j =1}^{J} A_j \bm{u}_j =\bm{ c }.
	\end{aligned}
\end{equation}

Note that the analysis in this paper focuses on the planar grain boundaries, for which each $\bm{u}_j$ is a constant vector.

\FloatBarrier
\section{ADMM algorithm and convergence analysis}
\label{sec:ADMM algorithm}
The objective function of the constrained minimization problem in Eq.~\eqref{eq:cp} is separable.
It is natural to consider  using ADMM \cite{Distributed} to solve it.
After modifying the ADMM algorithm with an increasing penalty parameter, we are able to
 prove the convergence of this modified ADMM algorithm for the constrained minimization problem for grain boundary dislocation structure.

\subsection{ADMM algorithm}
Similar to the ALM, ADMM converts the constrained minimization problem in Eq.~\eqref{eq:cp} into an unconstrained minimization problem, with the augmented Lagrangian function as the new objective function defined by
\begin{equation}\label{eq:Lp}
	L_{\rho }(\bm{u}_1,\dots,\bm{u}_J, w):
	=\sum_{j=1}^J f_j(\bm{u}_j) + w^{\T} (\sum_{j=1}^J A_j \bm{u}_j -\bm{ c }) + \frac{\rho }{2} \| \sum_{j=1}^J A_j \bm{u}_j -\bm{ c }\|^2,
\end{equation}
where $w \in \mathbb{R}^6 $ is the Lagrangian multiplier, $\rho >0 $ is the penalty parameter.

Compared to the penalty method and ALM,
ADMM simplifies the minimization process by
 splitting the variables into blocks and
converting nonconvex minimization in a high dimensional space into the minimization of subvariables sequentially in a lower dimensional space. The objective functions of these lower dimensional minimization are strictly convex and thus solvable.  Numerically, in the minimization of each block, we use gradient descent until certain stopping criteria are met.
That is, in each iteration, we obtain
\begin{equation}
 \bm{u}_j^{(k+1)} := \arg\min_{\bm{u}_j} L_{\rho^{(k)}}(\bm{u}_{i<j}^{(k+1)},\bm{u}_j,\bm{u}_{i>j}^{(k)},w^{(k)}),
\end{equation}
by updating $\bm{u}_j$ using gradient descent
\begin{equation}\label{eqn:gradient-descent}
    \bm{u}_j^{i+1} :=
    \bm{u}_j^{i} - \alpha
    \nabla_{\bm{u}_j} L_{\rho^{(k)}}(\bm{u}_{i<j}^{(k+1)},\bm{u}_j^i,\bm{u}_{i>j}^{(k)},w^{(k)}),
\end{equation}
where $\bm{u}_j^{i+1} $ is the value after the $i$-th step of gradient descent,
with initialization $\bm{u}_j^0= \bm{u}_j^{(k)}$ and learning step $\alpha >0$.
The Lagrangian multiplier $w$ is updated by the iteration scheme
\begin{equation}
 w^{(k+1)} := w^{(k)} + \rho^{(k)} (\sum_{j=1}^J A_j \bm{u}_j^{(k+1)}-\bm{ c }).
 \end{equation}

We modify the ADMM algorithm using an increasing penalty parameter, i.e.,
the penalty parameter $\rho$ is multiplied by a factor $\beta>1$ in each iteration:
\begin{equation}
\rho ^{(k+1)} := \beta \rho^{(k)}.
\end{equation}
This modified ADMM algorithm is summarized in Algorithm~\ref{alg:ADMM}.

\begin{algorithm}[Modified ADMM algorithm]
	\label{alg:ADMM}
	\begin{algorithmic}
        \hrule
		\REQUIRE Initialize $\bm{u}^{(0)}=(\bm{u}_1^{(0)}, \bm{u}_2^{(0)}, \dots, \bm{u}_J^{(0)} ) $ and  $
			w^{(0)}$,  proper $\rho^{(0)} >0, \beta>1, k=0 $.
		\WHILE{ stopping criteria not satisfied}
		\STATE  $\bm{u}_1^{(k+1)} := \arg\min_{\bm{u}_1} L_{\rho^{(k)}}(\bm{u}_1,\bm{u}_2^{(k)}, \dots,\bm{u}_J^{(k)},w^{(k)})$
                  using  Eq.~\eqref{eqn:gradient-descent}, \\ \ \ \ \ \ and same for other $\arg\min_{\bm{u}_j}$ problems.
		\STATE $\bm{u}_j^{(k+1)} := \arg\min_{\bm{u}_j} L_{\rho^{(k)}}(\bm{u}_{i<j}^{(k+1)},\bm{u}_j,\bm{u}_{i>j}^{(k)},w^{(k)}),  \quad j=2,\dots,J-1$,
		\STATE $\bm{u}_J^{(k+1)} := \arg\min_{\bm{u}_J} L_{\rho^{(k)}}(\bm{u}_1^{(k+1)},\dots,\bm{u}_{J-1}^{(k+1)}, \bm{u}_J,w^{(k)}), $
		\STATE $w^{(k+1)} := w^{(k)} + \rho^{(k)} (\sum_{j=1}^J A_j \bm{u}_j^{(k+1)}-\bm{ c })$,
		\STATE $\rho ^{(k+1)} := \beta \rho^{(k)}$,
		\STATE $k:=k+1$,
		\ENDWHILE
		\RETURN $\bm{u}=(\bm{u}_1^{(k)}, \bm{u}_2^{(k)} \dots, \bm{u}_J^{(k)})$.
	\end{algorithmic}
        \hrule
\end{algorithm}

\subsection{Convergence analysis: Main results}
To prove the convergence of  Algorithm~\ref{alg:ADMM} for this constrained minimization problem Eq.~\eqref{eq:cp},
we need the boundness assumption on the Lagrangian multiplier,
which was adopted in many convergence proofs of the ALM  \cite{Bertsekas} and ADMM  \cite{wang2018convergence,M2016,KeGuo,liu2023bregman} algorithms. 


\begin{theorem}\label{theorem}  
Assume the Lagrangian multiplier $w^{(k)}$ in Algorithm~\ref{alg:ADMM} for the constrained minimization problem in Eq.~\eqref{eq:cp} is bounded, i.e.,
there exists a positive constant $ M>0 $ such that
$\| w^{(k)} \|  \leq M $ for all $k>0$, then we have

(I) Algorithm~\ref{alg:ADMM} converges when solving the constrained minimization problem Eq.~\eqref{eq:cp}.
That is, the ADMM sequence $\bm{u}^{(k)} =(\bm{u}_1^{(k)}, \bm{u}_2^{(k)},
\dots, \bm{u}_J^{(k)} )$ converges to some point $\bm{u}^{\star} $.

(II) The limit $\bm{u}^{\star} $ is feasible
and $ \left\{ w^{(k)} \right\} $ converges. Denote the limit as $w^{\star}$.
Then $(\bm{u}^{\star}, w^{\star}) $ is a stationary point of  the augmented Lagrangian function in Eq.~\eqref{eq:Lp}.
\end{theorem}

Note that Theorem 1 shows that the limit point of the sequence generated by Algorithm 1 exists and is a stationary point of the augmented Lagrangian function. The stationary point corresponds to the minimizer of problem (10) when the equivalent unconstrained problem admits a unique stationary point, which can be further verified using Theorem 2 to be discussed in Section \ref{sec:unique}.

\begin{remark}
The boundedness of the multipliers can be justified by a  stability condition for an induced linear system associated with the coupled sequence $\{(\rho^{(k)}\bm{r}^{(k)}, w^{(k)})\}$ with $\bm{r}^{(k)}:=\sum_{j=1}^J A_j \bm{u}_j^{(k)}-c$. In particular, Proposition~\ref{pro:condition_for_w} below shows that if the corresponding iteration matrix $P$ satisfies the spectral radius $\rho(P)<1$, then the multiplier sequence $\{w^{(k)}\}$ is bounded. Here the matrix $P$ depends on the problem structure $A=[A_1,\dots,A_J]$, and the penalty growth factor $\beta$.
\end{remark}


\subsection{Convergence analysis: Proofs}

We first show some properties of Algorithm~\ref{alg:ADMM} and the constrained minimization problem Eq.~\eqref{eq:cp}.

\begin{lemma}\label{le:f}
$\| \nabla ^2 f_j \|$ are uniformly bounded on a bounded domain $U$, i.e.,
there exists a positive constant $C >0$, such that
$\| \nabla ^2 f_j \| \leq C$.
Then for $\forall x, y \in U$,
\begin{equation}\label{eq:f1}
	f_j(y)-f_j(x)- \langle \nabla f_j(x),y-x \rangle  \geq - C\| y-x\|^2.
\end{equation}
\end{lemma}


\begin{proof}
Because $f_j$ is twice continuously differentiable,
given any bounded closed subset,  $\| \nabla ^2 f_j \|$ is bounded.
For Eq.~\eqref{eq:f1},
	\begin{align*}
		f_j(y)-f_j(x)-\langle \nabla f_j(x),y-x \rangle &= (y-x)^\T \nabla^2 f_j(\xi)(y-x),\  \xi \text{ between } x\text{ and }y\\
		& \geq -\|\nabla^2 f_j (\xi)\| \cdot \|y-x \|^2\\
		& \geq - C \| y-x \|^2.
	\end{align*}
\hfill\qed
\end{proof}


\begin{lemma}\label{le:I}
Each $A_j$, ($A_j \in R^{6 \times 2} $; see Eq.~\eqref{eq:Aj}),  $j=1,2,\dots,J$, is a semi-orthogonal matrix \footnote{A real $m \times n$ matrix $A$ that satisfies $A A^\T= I_m$ or $A^\T A= I_n$ is called semi-orthogonal. }
multiplied by the positive constant $b$,
where $b$ is the length of a Burgers vector, i.e.,
\begin{equation}\label{eqn:semi-orthogonal}
A_j^{\T}A_j=b^2I_2.
\end{equation}
Moreover, we have
for any $x,y \in \mathbb{R}^2$,
	\begin{equation}
		\| A_{j} (y - x) \|^2 = b^2 \| y - x\|^2,  \ \  j=1,2,\dots,J.
	\end{equation}
\end{lemma}

\begin{proof}
We have $(\bm{b^{(j)} })^{\T} \bm{b^{(j)}} = b^2$ for $j=1,2,\dots, J$. Note that
all the Burgers vectors $ \bm{b^{(j)}}$'s have the same length $b$. Thus
	\begin{align*}
		\frac{1}{b^2} A_j^{\T}A_j
		&= \frac{1}{b^2} \left[
		\begin{array}{cc}
			(\bm{b}^{(j)})^{\T} & \bm{0} \\
			\bm{0}&  (\bm{b}^{(j)})^{\T}
		\end{array}
		\right]
		\left[
		\begin{array}{cc}
			\bm{b}^{(j)} & \bm{0} \\
			\bm{0}&  \bm{b}^{(j)}
		\end{array}
		\right]  \\
		&=\frac{1}{b^2} \left[
		\begin{array}{cc}
			(\bm{b}^{(j)} )^{\T} \bm{b}^{(j)}& 0 \\
			0&  (\bm{b}^{(j)})^\T \bm{b}^{(j)}
		\end{array}
		\right]  \\
		&=\frac{1}{b^2}
		\left[
		\begin{array}{cc}
			b^2 & 0 \\
			0&  b^2
		\end{array}
		\right]
		= I_2,
	\end{align*}
	and
	\begin{align*}
		\| A_j (y - x)\|^2
		&= (y-x)^{\T} A_j^{\T} A_j (y-x) \\
		&=b^2 (y-x)^{\T} (y-x)  \\
		&=b^2 \| y - x \|^2.
	\end{align*}
\hfill\qed
\end{proof}

\hspace{1pt}
\begin{proposition}[Subproblem convexity]\label{pro:strictL}
In each step of the $k$-th iteration in Algorithm~\ref{alg:ADMM}, when the penalty parameter $\rho^{(k)} $ is large enough,
$L_j(\bm{u}_j):= L_{\rho^{(k)}}(\bm{u}_{i<j}^{(k+1)},\bm{u}_j,\bm{u}_{i>j}^{(k)},w^{(k)}) $ with subvariables $\bm{u}_{i},i \neq j$,  fixed,  is a strictly convex function of  $ \bm{u}_j$ for $j=1,2,\dots,J$. \\
Therefore,  the minimization
	\begin{equation*}
		\bm{u}_j^{(k+1)} = \arg\min L_{\rho ^{(k)}}(\bm{u}_{i<j}^{(k+1)},\bm{u}_j,\bm{u}_{i>j}^{(k)},w^{(k)})
	\end{equation*}
is equivalent to
	\begin{equation*}
		\nabla _{\bm{u}_j}L_{\rho ^{(k)}}(\bm{u}_{i<j}^{(k+1)},\bm{u}_j^{(k+1)},\bm{u}_{i>j}^{(k)},w^{(k)})=\bm{0}.
	\end{equation*}
For example,
	$L_1(\bm{u}_1)= L_{\rho ^{(k)}}(\bm{u}_1, \bm{u}_{j>1}^{(k)}, w^{(k)} ) $ is a strictly convex function of $\bm{u}_1$.
\end{proposition}

\begin{proof}
Using Lemma~\ref{le:I}, for  $j=1,2,\dots,J$, we have
	\begin{equation}
		\nabla^2_{\bm{u}_j} L_j(\bm{u}_j) = \nabla^2 f_j(\bm{u}_j) +\rho^{(k)} A_j^\T A_j
		=\nabla^2 f_j(\bm{u}_j) +\rho^{(k)}b^2 I_2 .
	\end{equation}
Thus $ \nabla^2_{\bm{u}_j} L_j(\bm{u}_j)$ is positive definite when $\rho^{(k)} $ is large enough.
Hence, \\
$L_{\rho^{(k)}}(\bm{u}_{i<j}^{(k+1)},\bm{u}_j,\bm{u}_{i>j}^{(k)},w^{(k)})$ is a strictly convex function of $\bm{u}_j$.
\hfill\qed
\end{proof}

 
\begin{proposition}[Stability condition for multiplier boundedness]\label{pro:condition_for_w}
Under Algorithm~\ref{alg:ADMM} and Lemmas~\ref{le:f}-\ref{le:I}, 
define 
$\bm{r}^{(k)}:=\sum_{j=1}^J A_j \bm{u}_j^{(k)}-c,$
and let
\begin{equation}
A=[A_1,\dots,A_J]\in\mathbb{R}^{6\times 2J},
\end{equation} and
\begin{equation}\label{eqn:Z}
Z:=
\begin{bmatrix}
A_1^T A_1 & 0 & \cdots & 0\\
A_2^T A_1 & A_2^T A_2 & \cdots & 0\\
\vdots & \vdots & \ddots & \vdots\\
A_J^T A_1 & A_J^T A_2 & \cdots & A_J^T A_J
\end{bmatrix}\in\mathbb{R}^{2J\times 2J}
\end{equation}
be a block lower triangular matrix.
Define
\begin{equation}\label{eqn:P}
P=
\begin{bmatrix}
\beta(I_6-AZ^{-1}A^T) & -\beta AZ^{-1}A^T\\
(I_6-AZ^{-1}A^T) & (I_6-AZ^{-1}A^T)
\end{bmatrix}\in\mathbb{R}^{12\times 12},
\end{equation}
with the penalty growth factor $\beta$.
If the spectral radius $\rho(P)<1$, then the multiplier sequence $\{w^{(k)}\}$ is bounded.
\end{proposition}

\begin{proof} 
For the general $J$-block problem, 
the augmented Lagrangian function is 
$L_{\rho }(\bm{u}_1,\dots,\bm{u}_J, w)
	=\sum_{j=1}^J f_j(\bm{u}_j) + w^{\T} (\sum_{j=1}^J A_j \bm{u}_j -\bm{ c }) + \frac{\rho }{2} \| \sum_{j=1}^J A_j \bm{u}_j -\bm{ c }\|^2$.
By the optimality condition of the $j$-th subproblem in Algorithm 1, we have 
\begin{equation}\label{eq:stationary} 
\nabla f_j(\bm{u}_{j}^{(k+1)}) +A_j^\T w^{(k)}+ \rho^{(k)} A_j^\T(
\sum_{i \leq j} A_i\bm{u}_i^{(k+1)} 
+\sum_{i> j} A_i\bm{u}_i^{(k)}-\bm{c})   = \bm{0}. 
\end{equation}  
Denote 
$\nabla f_j^{(k+1)} :=\nabla f_j (\bm{u}_{j}^{(k+1)})$
and
$\bm{r}^{(k)}:= \sum_{j=1}^{J} A_j\bm{u}_j^{(k)}-\bm{c} $.
Then the above equations can be rewritten as  
\begin{equation}
\rho^{(k)} A_j^\T 
\sum_{i\leq j}   A_i(\bm{u}_i^{(k+1)}-\bm{u}_i^{(k)}) 
= -(\nabla f_j^{(k+1)} +A_j^\T w^{(k)})- \rho^{(k)} A_j^\T \bm{r}^{(k)}, j=1,\dots,J.
\end{equation}  
Stacking these equations together yields
\begin{equation}\label{eqn:diff-uk}
\rho^{(k)} Z (\bm{u}^{(k+1)}-\bm{u}^{(k)})
=
-G^{(k+1)} -A^\T w^{(k)}-\rho^{(k)}A^\T \bm{r}^{(k)},
\end{equation}
where
$ \bm{u}^{(k)}= 
\left[
\begin{array}{c}
 \bm{u}_1^{(k)} \\
 \bm{u}_2^{(k)} \\
\vdots \\
 \bm{u}_J^{(k)}    \\
\end{array}
\right]
$,
$G^{(k)}:=\left[
\begin{array}{c}
\nabla f_1^{(k)}  \\
\nabla f_2^{(k)}  \\
\vdots  \\
\nabla f_J^{(k)}     \\
\end{array}
\right]$,
 and $Z$ is defined in Eq.~\eqref{eqn:Z}. 

We first show that the matrix $Z$ is invertible.
By the semi-orthogonality condition, we have
\[
A_i^\T A_j=\gamma_{ij}I_2,
\quad \gamma_{ij}:=(\bm{b}^{(i)})^\T \bm{b}^{(j)},
\quad \gamma_{ii}:=(\bm{b}^{(i)})^\T \bm{b}^{(i)}=b^2>0, 
\quad i,j=1,\dots, J.
\]
The block lower triangular matrix \(Z\) can be written as
$Z=L\otimes I_2,$
where  
\begin{equation}\label{eqn:L}
L=
\begin{bmatrix}
\gamma_{11} & 0 & \cdots & 0\\
\gamma_{21} & \gamma_{22} & \cdots & 0\\
\vdots & \vdots & \ddots & \vdots\\
\gamma_{J1} & \gamma_{J2} & \cdots & \gamma_{JJ}
\end{bmatrix},
\end{equation}
and $\otimes $ denotes the Kronecker product.
Since 
$\gamma_{ii}=b^2>0$,  $i=1,\ldots,J$, and $L$ and $Z$ are lower triangular, both $L$ and $Z$ are invertible, both $L^{-1}$ and $Z^{-1}$ are lower triangular, and
\begin{equation}
Z^{-1}=L^{-1}\otimes I_2. 
\end{equation}
It is easy to compute that
\begin{equation}
 L^{-1}=(h_{ij})_{i,j=1}^J, 
\end{equation}
where the diagonal entries are
$h_{ii}=\frac{1}{b^2}$, 
and for \(i>j\),  
$h_{ij}
=
-\frac{1}{b^2}
\sum_{\ell=j}^{i-1}\gamma_{i\ell}h_{\ell j}$.

Now using $\bm{r}^{(k+1)} 
=\bm{r}^{(k)} +  A(\bm{u}^{(k+1)}-\bm{u}^{(k)})$,  and solving  for $\bm{u}^{(k+1)}-\bm{u}^{(k)}$ from Eq.~\eqref{eqn:diff-uk}, 
we obtain
$$ \bm{r}^{(k+1)} 
=
(I_6-  AZ^{-1} A^\T) \bm{r}^{(k)}
-\frac{1}{\rho^{(k)}  } AZ^{-1} A^\T w^{(k)}
-\frac{1}{\rho^{(k)}  } AZ^{-1}G^{(k+1)} . $$
Multiplying both sides by $\rho^{(k+1)}=\beta\rho^{(k)}$ gives
\begin{equation}\label{eqn:rk}
\rho^{(k+1)} \bm{r}^{(k+1)}  
=
(I_6-  AZ^{-1} A^\T) \beta\rho^{(k)}\bm{r}^{(k)} 
-  \beta  AZ^{-1} A^\T w^{(k)}
- \beta  AZ^{-1}G^{(k+1)} .
\end{equation}

Moreover, by the multiplier update
$w^{(k+1)} =w^{(k)} +  \rho^{(k)}\bm{r}^{(k+1)} $,
and using Eq.~\eqref{eqn:rk}, we have
\begin{equation}\label{eqn:wk}
w^{(k+1)}  
=
(I_6-  AZ^{-1} A^\T) \rho^{(k)}\bm{r}^{(k)}
+(I_6-  AZ^{-1} A^\T ) w^{(k)}
-  AZ^{-1}G^{(k+1)}. 
\end{equation}

Combining the  two relations \eqref{eqn:rk} and \eqref{eqn:wk}, we obtain the coupled recursion
\begin{equation}
\begin{bmatrix}
\rho^{(k+1)}\bm{r}^{(k+1)}\\
w^{(k+1)}
\end{bmatrix}
=
P
\begin{bmatrix}
\rho^{(k)}\bm{r}^{(k)}\\
w^{(k)}
\end{bmatrix}
-
\begin{bmatrix}
\beta I_6\\
I_6
\end{bmatrix}
AZ^{-1}G^{(k+1)},
\end{equation}
where  matrix $P$ is given in Eq.~\eqref{eqn:P}. 
Iterating this recursion gives
\begin{equation}
\begin{bmatrix}
\rho^{(k)}\bm{r}^{(k)}\\
w^{(k)}
\end{bmatrix}
=
P^k
\begin{bmatrix}
\rho^{(0)}\bm{r}^{(0)}\\
w^{(0)}
\end{bmatrix}
-
\sum_{\ell=0}^{k-1}
P^{k-1-\ell}
\begin{bmatrix}
\beta I_6\\
I_6
\end{bmatrix}
AZ^{-1}G^{(\ell+1)}.
\end{equation}

When $\rho(P)<1$, there exist constants $C_0>0$ and $\tau\in(0,1)$ such that
\[
\|P^k\|\le C_0\tau^k,\qquad k\ge 0.
\]
Since $\nabla f_j$ is bounded, there exists $C_1>0$ such that
\[
\|AZ^{-1}G^{(k)}\|\le C_1,\qquad \forall k\ge 0.
\]
Therefore,
\[
\left\|
\begin{bmatrix}
\rho^{(k)}\bm{r}^{(k)}\\
w^{(k)}
\end{bmatrix}
\right\|
\le
C_0\tau^k
\left\|
\begin{bmatrix}
\rho^{(0)}\bm{r}^{(0)}\\
w^{(0)}
\end{bmatrix}
\right\|
+
C_0C_1\left\|
\begin{bmatrix}
\beta I_6\\
I_6
\end{bmatrix}
\right\|
\sum_{\ell=0}^{k-1}\tau^{k-1-\ell}.
\]
The geometric sum on the right-hand side is uniformly bounded in $k$, hence
\[
\sup_{k\ge 0}
\left\|
\begin{bmatrix}
\rho^{(k)}\bm{r}^{(k)}\\
w^{(k)}
\end{bmatrix}
\right\|
<\infty.
\]
Thus both $\{\rho^{(k)}\bm{r}^{(k)}\}$ and $\{w^{(k)}\}$ are bounded.
\hfill\qed
\end{proof}

\begin{remark}[Verification of the spectral condition \(\rho(P)<1\) in Proposition~\ref{pro:condition_for_w}].

In practice, given the penalty growth factor $\beta$, we can numerically compute the eigenvalues of   \(P\in\mathbb{R}^{12\times 12}\) and 
verify if the spectral condition  \(\rho(P)<1\) holds. 
 
Alternatively,  the eigenvalues of   \(P\) can also be calculated using the corresponding eigenvalues of a smaller matrix  and the penalty growth factor $\beta$.

Let
\[
M:=AZ^{-1}A^\T\in\mathbb{R}^{6\times 6},\qquad S:=I_6-M .
\]
 It can be calculated that
\[
M=AZ^{-1}A^\T
=
A(L^{-1}\otimes I_2)A^\T
=
\sum_{i=1}^J\sum_{j=1}^{i} h_{ij}A_iA_j^\T .
\] 
Let \(\sigma(M)\) denote the spectrum of \(M\),  which is the set
of all eigenvalues of \(M\).
By the spectral mapping theorem,
\(\sigma(S)=\{1-\mu:\mu\in\sigma(M)\}\).
Thus, if \(\mu\in\sigma(M)\), the corresponding eigenvalue of \(S\) is
$s=1-\mu.$

The matrix $P$ can be written as
\[
P=
\begin{bmatrix}
\beta S & -\beta M\\
S & S
\end{bmatrix}
=
\begin{bmatrix}
\beta S & -\beta(I_6-S)\\
S & S
\end{bmatrix}.
\]
Since all blocks are polynomials in \(S\), they commute. Therefore, an eigenvalue $\lambda$ of $P$ satisfies
\[
\det(\lambda I_{12}-P)
=
\det\left(
\lambda^2I_6-(\beta+1)\lambda S+\beta S
\right).
\]
Hence, for each \(s\in\sigma(S)\), the corresponding two eigenvalues of \(P\)
satisfy
\[
\lambda^2-(\beta+1)s\lambda+\beta s=0.
\]
Using \(s=1-\mu\), this can be written as
\begin{equation}\label{eqn:lambda}
\lambda^2-(\beta+1)(1-\mu)\lambda+\beta(1-\mu)=0.
\end{equation}
Therefore, we have
\begin{equation}\label{eq:eigen_P}
\lambda_{\pm}(\mu)
=
\frac{
(\beta+1)(1-\mu)
\pm
\sqrt{
(\beta+1)^2(1-\mu)^2-4\beta(1-\mu)
}
}{2},
\end{equation}
and the spectral radius of \(P\) can be expressed as
\begin{equation}
\rho(P)
=
\max_{\mu\in\sigma(M)}
\max\left\{
|\lambda_+(\mu)|,\,
|\lambda_-(\mu)|
\right\}.
\end{equation}


When all the eigenvalues \(\mu\in\sigma(M)\) are real,
the above spectral condition admits a simple interval form. In fact, for a 
real \(\mu\), 
all the coefficients of the characteristic equation \eqref{eqn:lambda} are real,
by the Jury stability criterion, both roots  satisfy \(|\lambda|<1\) if and only if
$1\pm(\beta+1)(1-\mu)+\beta(1-\mu)>0$, and
$1-|\beta(1-\mu)|>0$.
Since \(\beta > 1\), the if and only if conditions for \(\rho(P)<1\) are
\begin{equation}
1-\frac{1}{\beta}
<
\mu
<
1+\frac{1}{2\beta+1}.
\end{equation}
\end{remark}




Now we prove Theorem~\ref{theorem}.

\begin{proof}[Theorem~\ref{theorem} Part I]
Based on Algorithm~\ref{alg:ADMM},
we calculate the change of value of the augmented Lagrangian function in one iteration.
To be specific, for the $k$-th iteration,  we first divide the iteration into three parts: (i) updates of the first $J-1$ blocks, (ii) updates of the $J$-th block and the Lagrangian multiplier $w$, and (iii) update of the penalty parameter $\rho$.
We calculate the differences in these three parts separately and finally sum them up.
	
Recall the augmented Lagrangian function defined in Eq.~\eqref{eq:Lp}:
	\begin{equation*}
		L_{\rho }(\bm{u}_1,\dots,\bm{u}_J, w)
		=\sum_{j=1}^J f_j(\bm{u}_j) + w^{\T} (\sum_{j=1}^J A_j \bm{u}_j -\bm{ c }) + \frac{\rho }{2} \| \sum_{j=1}^J A_j \bm{u}_j -\bm{ c } \|^2.
	\end{equation*}
From direct calculations and Proposition~\ref{pro:strictL},
the difference results from the updating from
$ (\bm{u}_1^{(k)},\dots,\bm{u}_{J-1}^{(k)})$
to
$(\bm{u}_1^{(k+1)},\dots,\bm{u}_{J-1}^{(k+1)})$ is
	\begin{equation}\label{eq:A1}
		\begin{aligned}
			&  L_{\rho ^{(k)}}(\bm{u}_1^{(k)},\dots,\bm{u}_{J}^{(k)},w^{(k)}) - L_{\rho^{(k)}}(\bm{u}_1^{(k+1)},\dots,\bm{u}_{J-1}^{(k+1)}, \bm{u}_J^{(k)}, w^{(k)})   \\
			 =&
			\sum_{j=1}^{J-1} \left[f_j(\bm{u}_{j}^{(k)})-f_j(\bm{u}_{j}^{(k+1)})
			-\langle \nabla f_j(\bm{u}_{j}^{(k+1)}),\bm{u}_{j}^{(k)}-\bm{u}_{j}^{(k+1)} \rangle \right] \\
			&+ \sum_{j=1}^{J-1} \frac{\rho ^{(k)}}{2}
			\| A_{j}\bm{u}_{j}^{(k+1)} -A_{j} \bm{u}_{j}^{(k)} \|^2.
		\end{aligned}
	\end{equation}
Secondly, the difference due to the updating from
	$(\bm{u}_{J}^{(k)},w^{(k)}) $ to
	$(\bm{u}_{J}^{(k+1)},w^{(k+1)})$ is
	\begin{equation}\label{eq:A2}
		\begin{aligned}
			&  L_{\rho^{(k)}}(\bm{u}_1^{(k+1)},\dots,\bm{u}_{J-1}^{(k+1)}, \bm{u}_J^{(k)}, w^{(k)}) - L_{\rho^{(k)}}(\bm{u}_1^{(k+1)},\dots,\bm{u}_{J}^{(k+1)},w^{(k+1)}) \\
			 =&f_J(\bm{u}_J^{(k)} )-f_J(\bm{u}_J^{(k+1)})-\langle \nabla f_J(\bm{u}_J^{(k+1)}),\bm{u}_J^{(k)}-\bm{u}_J^{(k+1)}\rangle \\
			&+  \frac{\rho ^{(k)} }{2} \| A_J \bm{u}_J^{(k+1)} -A_J \bm{u}_J^{(k)} \|^2
			-\frac{1}{\rho^{(k)}} \| w^{(k+1)} - w^{(k)} \|^2.
		\end{aligned}
	\end{equation}
Finally, the difference due to the updating from
	$\rho^{(k)}$ to $\rho^{(k+1)}$ is
	\begin{equation}\label{eq:A3}
		\begin{aligned}
			& L_{\rho^{(k)}}(\bm{u}_1^{(k+1)},\dots,\bm{u}_{J}^{(k+1)},w^{(k+1)}) - L_{\rho^{(k+1)}}(\bm{u}_1^{(k+1)},\dots,\bm{u}_{J}^{(k+1)},w^{(k+1)})  \\
			 =&\frac{1- \beta}{2} \rho^{(k)} \| \sum_{j=1}^{J} A_j \bm{u}_j^{(k+1)} -\bm{c} \|^2   \\
			=&\frac{1- \beta}{2 \rho^{(k)} } \| w^{(k+1)} - w^{(k)} \|^2.
		\end{aligned}
	\end{equation}

We would like to remark that under different problem assumptions with constant  penalty parameter, the updates of the augmented Lagrangian function after each iteration have been considered in the convergence proof in Ref.~\cite{Yin} and they were able to show that the augmented Lagrangian function is monotonically decreasing during the ADMM iteration process.
	
Summing up Eqs.~\eqref{eq:A1}, \eqref{eq:A2} and \eqref{eq:A3},  
under the assumption of boundedness of multipliers
and Lemma~\ref{le:f}, we have
	\begin{equation}\label{eq:A4}
		\begin{aligned}
			& L_{\rho^{(k)}}(\bm{u}_1^{(k)},\dots,\bm{u}_{J}^{(k)},w^{(k)}) - L_{\rho^{(k+1)}}(\bm{u}_1^{(k+1)},\dots,\bm{u}_{J}^{(k+1)},w^{(k+1)}) \\
			=&\sum_{j=1}^{J} \left[f_j(\bm{u}_j^{(k)} )-f_j(\bm{u}_j^{(k+1)})-\langle \nabla f_j(\bm{u}_j^{(k+1)}),\bm{u}_j^{(k)}-\bm{u}_j^{(k+1)} \rangle \right]  \\
			&+ \sum_{j=1}^{J} \frac{\rho^{(k)}}{2} \| A_j \bm{u}_j^{(k+1)} -A_j \bm{u}_j^{(k)} \|^2
			-\frac{\beta + 1}{2}\frac{1}{\rho^{(k)}} \| w^{(k+1)} - w^{(k)} \|^2 \\
			 \geq& \left(\frac{\rho^{(k)}}{2} b^2 - C\right) \sum_{j=1}^{J}
			\| \bm{u}_j^{(k+1)} - \bm{u}_j^{(k)}\|^2- \frac{1}{\rho^{(k)}} ( \frac{\beta + 1}{2}  4M^2 )  \\
			=& \left(\frac{\rho^{(k)}}{2} b^2 - C\right)  \| \bm{u}^{(k+1)} - \bm{u}^{(k)} \|^2 -\frac{\delta}{\rho^{(k)}},
		\end{aligned}
	\end{equation}
where $\rho^{(k)}$ is large enough.
	
For simplicity, let $\rho^{(k)}$ be large enough such that
$ \frac{\rho^{(k)}}{2} b^2 - C \geq  1$, and denote the constant $ \delta := 2(\beta + 1) M^2. $
Then Eq.~\eqref{eq:A4} becomes
	\begin{equation}\label{eq: add A1-4}
		\begin{aligned}
			& L_{\rho^{(k)}}(\bm{u}_1^{(k)},\dots,\bm{u}_{J}^{(k)},w^{(k)}) - L_{\rho^{(k+1)}}(\bm{u}_1^{(k+1)},\dots,\bm{u}_{J}^{(k+1)},w^{(k+1)}) \\
			\geq& \sum_{j=1}^{J} \| \bm{u}_j^{(k+1)} - \bm{u}_j^{(k)} \|^2  -\frac{\delta}{\rho^{(k)}} \\
			=& \| \bm{u}^{(k+1)} - \bm{u}^{(k)} \|^2 -\frac{\delta}{\rho^{(k)}} .
		\end{aligned}
	\end{equation}

For some $N_1 >0 $ such that $\rho^{(N_1)}$ is large enough and Eq.~\eqref{eq: add A1-4} holds for $k=N_1$.
Then Eq.~\eqref{eq: add A1-4} should also hold for $k \geq N_1$, since $\rho^{(k)}=\beta^k \rho^{(0)}$ with $\beta>1$.
Summing up Eq.~\eqref{eq: add A1-4} for the superscript $k$ from $N_1$ to $N$ with $ N> N_1, $ we obtain that
	\begin{equation}\label{eq:0-N}
		\begin{aligned}
			& L_{\rho^{(N_1)}}(\bm{u}_1^{(N_1)},\dots,\bm{u}_{J}^{(N_1)},w^{(N_1)}) - L_{\rho^{(N+1)}}(\bm{u}_1^{(N+1)},\dots,\bm{u}_{J}^{(N+1)},w^{(N+1)}) \\
			 = &\sum_{k=N_1}^{N} \left[ L_{\rho^{(k)}}(\bm{u}_1^{(k)},\dots,\bm{u}_{J}^{(k)},w^{(k)}) - L_{\rho^{(k+1)}}(\bm{u}_1^{(k+1)},\dots,\bm{u}_{J}^{(k+1)},w^{(k+1)}) \right]  \\
			 \geq& \sum_{k=N_1}^{N} \| \bm{u}^{(k+1)} - \bm{u}^{(k)} \|^2 - \sum_{k=N_1}^{N}  \frac{\delta}{\rho^{(k)}}.
		\end{aligned}
	\end{equation}
Thus $\sum_{k=N_1}^{N}  \| \bm{u}^{(k+1)} - \bm{u}^{(k)} \|^2 $ is bounded (increasing, hence convergent), because $L_{\rho }(\bm{u}_1,\dots,\bm{u}_{J},w )$ has lower bound and
	\[ \sum_{k=N_1}^{N}  \frac{\delta}{\rho^{(k)} }
	\leq  \frac{\delta}{\rho^{(N_1)} } \frac{1}{1-\frac{1}{\beta}}\] with $ \rho^{(N_1)} = \rho^{(0)} \beta^{N_1}$.

	
Therefore, the series $\sum_{k=0}^{\infty}  \| \bm{u}^{(k+1)} - \bm{u}^{(k)} \|^2 $ is convergent.
Then $\left\{\bm{u}^{(k)}  \right\}$ is a convergent sequence,
denote the limit by
	\begin{equation}
		\bm{u}^{\star}
		=(\bm{u}_1^{\star},\bm{u}_2^{\star},\dots,\bm{u}_{J}^{\star}) =
		\lim_{k \rightarrow \infty} (\bm{u}_1^{(k)},\bm{u}_2^{(k)},\dots,\bm{u}_{J}^{(k)}).
	\end{equation}
\hfill\qed
\end{proof}

\hspace{1pt}
\begin{proof}[Theorem~\ref{theorem} Part II]
First, we prove that the limit of ADMM $\bm{u}^{\star}$ is feasible.
From the update of the Lagrangian multiplier $w$ in Algorithm~\ref{alg:ADMM}:
	\begin{equation}
		w^{(k+1)} = w^{(k)} + \rho^{(k)} (\sum_{j=1}^{J} A_j \bm{u}_j^{(k+1)} -\bm{c}),
	\end{equation}
we obtain that for any positive integer $K$,
	\begin{equation}
		w ^{(K)} = w^{(0)} + \sum_{k=0}^{K} \rho^{(k)} (\sum_{j=1}^{J} A_j \bm{u}_j^{(k+1)} -\bm{c}) .
	\end{equation}
Since $\rho ^{(k)}$ is increasing and 
under the assumption 
that $w^{(K)}$ is bounded, we have
	\begin{equation*}
		\lim_{k\rightarrow \infty} \sum_{j=1}^{J} A_j \bm{u}_j^{(k)} -\bm{c}  = \bm{0},
	\end{equation*}
which implies
	\begin{equation}
		\sum_{j=1}^{J} A_j \bm{u}_j^{\star} -\bm{c} = \bm{0}.
	\end{equation}
Hence $\bm{u}^{\star}$ satisfies the linear constraint in Eq.~\eqref{eq:cp}, thus it is feasible,
and $ \left\{ w^{(k)} \right\} $ converges. Denote
$ w^{\star} =\lim_{k\rightarrow \infty} w^{(k)} $.
For feasible points,
the augmented Lagrangian function  has the same value when the penalty parameter $\rho$ takes difference values.

Moreover,  we have that $(\bm{u}^{\star}, w^{\star})$ satisfies
	\begin{equation}
		\begin{cases}
			\bm{u}_j^{\star} = \arg\min_{\bm{u}_j} L_{\rho^{(\infty)}}(\bm{u}_{i<j}^{\star},\bm{u}_j,\bm{u}_{i>j}^{\star},w^{\star} ), \ \ j=1,\dots,J,\\
			\sum_{j=1}^{J}  A_j \bm{u}_j^{\star} -\bm{c} =\bm{0}.
		\end{cases}
	\end{equation}
Using Proposition~\ref{pro:strictL}, we have
	\begin{equation}
		\begin{cases}
			\nabla_{\bm{u}_j} L_{\rho} (\bm{u}^{\star}, w^{\star}) 	=\bm{0}, j=1,\dots,J,\\
			\nabla_{w} L_{\rho} (\bm{u}^{\star}, w^{\star})  =\sum_{j=1}^{J} A_j \bm{u}_j^{\star} -\bm{c} =\bm{0}.
		\end{cases}
	\end{equation}
These results mean that the limit $(\bm{u}^{\star}, w^{\star})  $ is a stationary point of $L_{\rho}(\bm{u}, w)$.
\hfill\qed
\end{proof}

\subsection{Discussion}
\paragraph{Assumptions in other convergence proofs of ADMM for nonconvex problems.}
In Ref.~\cite{Yin}, the following optimization problem
\begin{equation}
	\begin{aligned}
		\min_{x_0,x_1,\dots,x_p,y} & \phi(x_0,x_1,\dots,x_p,y)  \\
		\text{ s.t. } & A_0 x_0+ A_1 x_1+\dots+ A_p x_p +By =b,
	\end{aligned}
\end{equation}
was considered, and
the assumption that $\img(A) \subseteq \img(B)$ was used to prove the convergence of ADMM, where $A = (A_0, A_1, \dots, A_p)$ is the coefficient matrix of variables $x_0, x_1,\dots, x_p$ except the last variable $y$,
$B$ is the coefficient matrix of the last variable $y$,  and $\img(\cdot)$ is the image of a matrix.
In Ref.~\cite{Hong}, convergence of the ADMM for solving the following nonconvex consensus and sharing problems was analyzed:
\begin{equation}
	\begin{aligned}
		\min & \sum_{k=1}^K g_k(\bm{x}_k) + l(\bm{x}_0)  \\
		\text{ s.t. }& \sum_{k=1}^K A_k\bm{x}_k=\bm{x}_0, \bm{x}_k \in X_k, k=1, \dots, K,
	\end{aligned}
\end{equation}
where $\bm{x}_k \in \mathbb{R}^{N_k}, A_k \in \mathbb{R}^{M\times N_k},\bm{x}_0 \in \mathbb{R}^{M}$.
in which the assumption that $A_k$ is of full column rank is required to make the minimization problem of $\bm{x}_k$ strongly convex.
It was pointed out in Ref.~\cite{Yin} that their assumptions  cover the setting in Ref.~\cite{Hong} since $\img(A_1, \dots, A_K) \subset \img(I_M)$.
Constant penalty parameter was used in these ADMM convergence proofs.
Ref.~\cite{yashtini2021multi}  assumed $\text{img}(A) \subseteq \text{img}(B)$ where $A,B$ are coefficient matrices of block variable $x,y$.
On the other hand, our minimization problem in Eq.~\eqref{eq:cp} does not satisfy the assumptions in the above ADMM convergence proofs. In our minimization problem,
we have
$\img(\left[A_1,\dots, A_{J-1}\right]) \not\subseteq \img(A_J)$,
and $A_j \in \mathbb{R}^{6\times 2}, j=1,\dots, J$, are of full column rank, and especially, each $A_j$ is semi-orthogonal up to a constant; see Eq.~\eqref{eqn:semi-orthogonal}.
Ref.~\cite{chen2025convergence} focused on the convergence of ADMM for solving linearly constrained optimization problems whose objective function is the sum of one weakly convex and two strongly convex functions;
Ref.~\cite{yang2019inexact} considered the optimization problem with the objective function $h(\bm{x}) = f(\bm{x}_1, \dots, \bm{x}_K) + \sum_{k=1}^{K} g_k(\bm{x}_k)$, where $f$ is smooth but not convex necessarily, $g_k$ is convex, 
while our objective function is nonconvex.
In Ref.~\cite{hien2024multiblock}, the analysis requires $\lambda_{\min}(BB^{\star})>0$, where $B$ is a linear map. The assumption does not hold in our setting since the constraint matrices $A_j$ are not jointly full rank.
In Ref.~\cite{wang2018convergence},
one of assumptions is that there is $\sigma>0$ such that $\sigma \|x\|^2 \leq \|C^T x\|^2, \forall x \in \mathbb{R}^m$,
where $ C\in \mathbb{R}^{m \times n_3}$ is the coefficient matrix of the last (third) variable block in the optimization problem. 
Correspondingly, in our problem, the coefficient matrix of the last variable block is $ A_J\in \mathbb{R}^{6 \times 2}$,$ A_J^T \in \mathbb{R}^{2 \times 6}$.
The condition that $\sigma \|x\|^2 \leq \|A_J^T x\|^2, \forall x \in \mathbb{R}^6$ does not hold.

As a result, our problem lies outside the scope of standard ADMM convergence theory and motivates the new analysis presented in this work.

\paragraph{Increasing penalty parameter.}
It was discussed in  Ref.~\cite{Distributed} that using different penalty parameters $\rho^{(k)}$
for each iteration of ADMM may improve the convergence in practice,
and may make performance less dependent on the initial choice
of the penalty parameter.
It was proposed in Ref.~\cite{Aybat} to use an ADMM algorithm with an increasing sequence of penalties to a solve a nonsmooth but convex optimization problem.

We would like to remark  that the divergent example of the multi-block ADMM constructed in Ref.~\cite{Caihua} can be made convergent if it is solved by the modified ADMM with an increasing penalty parameter.
Consider the divergent example in Ref.~\cite{Caihua}:
the objective function is null, and the constraint is a linear homogeneous equation with three variables:
\begin{equation}\label{eq:constraint}
	A_1 x_1+A_2 x_2 +A_3 x_3 =0,
\end{equation}
where $A_i \in \mathbb{R}^3, i=1,2,3,$ are column vectors, and the matrix $A =(A_1, A_2, A_3) $
is nonsingular.
The unique solution of Eq.~\eqref{eq:constraint} is $x_1=x_2=x_3=0$.
The corresponding augmented Lagrangian function is
\begin{equation}\label{eq:Lag}
	L_{\rho }(x_1, x_2, x_3, w)
	= w^{\T} (A_1 x_1+A_2 x_2 +A_3 x_3) + \frac{\rho }{2} \|A_1 x_1+A_2 x_2 +A_3 x_3 \|^2.
\end{equation}
If we apply our modified ADMM, analogous to Algorithm~\ref{alg:ADMM}, to solve the minimization problem, we have the following iteration formulation
\begin{equation}\label{eq:ite}
	\begin{cases}
		&	A_1^\T w^{(k)} +\rho^{(k)}  A_1^\T(A_1 x_1^{(k+1)} +A_2 x_2^{(k)}  +A_3 x_3^{(k)} ) =0, \\
		&	A_2^\T w^{(k)} +\rho^{(k)}  A_2^\T(A_1 x_1^{(k+1)} +A_2 x_2^{(k+1)}  +A_3 x_3^{(k)} ) =0, \\
		&	A_3^\T w^{(k)} +\rho^{(k)}  A_3^\T(A_1 x_1^{(k+1)} +A_2 x_2^{(k+1)}  +A_3 x_3^{(k+1)} ) =0, \\
		&	w^{(k+1)} -w^{(k)} +\rho^{(k)} (A_1 x_1^{(k+1)} +A_2 x_2^{(k+1)}  +A_3 x_3^{(k+1)} ) =0,\\
		&	\rho^{(k+1)} = \beta \rho^{(k)}, \beta >1.
	\end{cases}
\end{equation}
Reformulate the iteration in Eq.~\eqref{eq:ite} as
\begin{equation}
	x_1^{(k+1)} = \frac{1}{A_1^\T A_1}(-A_1^\T A_2 x_2^{(k)} -A_1^\T A_3 x_3^{(k)} + A_1^\T w^{(k)}/\rho^{(k)}),
\end{equation}
and
\begin{equation}
	\left(\begin{array}{c}
		x_2^{(k+1)}\\
		x_3^{(k+1)}\\
		w^{(k+1)}
	\end{array}\right)
	=M \left(\begin{array}{c}
		x_2^{(k)}\\
		x_3^{(k)}\\
		w^{(k)}
	\end{array}\right)
	= \dots
	=M^{k+1}
	\left(\begin{array}{c}
		x_2^{(0)}\\
		x_3^{(0)}\\
		w^{(0)}
	\end{array}\right) ,
\end{equation}
where $$ M=L^{-1}R, $$
\begin{equation}
	L= \left(\begin{array}{ccc}
		A_2^\T A_2 & 0 & 0_{1\times 3} \\
		A_3^\T A_2 & A_3^\T A_3 & 0_{1\times 3} \\
		A_2 & A_3 & \beta I_{3\times 3}
	\end{array}\right),
\end{equation}
\begin{equation}
	R=\left(\begin{array}{ccc}
		0 & -A_2^\T A_3 & A_2^\T\\
		0 & 0 & A_3^\T \\
		0_{3\times 1} & 0_{3\times 1} & I_{3\times 3}
	\end{array}\right)
	-\frac{1}{A_1^\T A_1}
	\left(\begin{array}{c}
		A_2^\T A_1\\
		A_3^\T A_1 \\
		A_1
	\end{array}\right)
	(-A_1^\T A_2, -A_1^\T A_3, A_1^\T).
\end{equation}

In Ref.~\cite{Caihua},  the matrix $A$ is constructed as
\begin{equation}
	A =(A_1, A_2, A_3)=
	\left( \begin{array}{ccc}
		1 & 1 & 1\\
		1 & 1 & 2\\
		1 & 2 & 2
	\end{array}
	\right),
\end{equation}
such that the spectral radius $\sigma(M)>1$ and then the ADMM interation is divergent.
However, with an increasing penalty parameter,
we have, for the same linear constraint,
\begin{equation}
	L= \left(\begin{array}{ccccc}
		6 & 0 & 0 & 0 & 0 \\
		7 & 9 & 0 & 0 & 0 \\
		1 & 1 & \beta & 0 & 0\\
		1 & 2 & 0 & \beta & 0\\
		2 & 2 & 0 & 0 & \beta
	\end{array}\right),
	L^{-1}= \frac{1}{54}
	\left(\begin{array}{ccccc}
		9 & 0 & 0 & 0 & 0 \\
		-7 & 6 & 0 & 0 & 0 \\
		-\frac{2}{\beta} & -\frac{6}{\beta} & \frac{54}{\beta} & 0 & 0\\
		\frac{5}{\beta} & -\frac{12}{\beta} & 0 & \frac{54}{\beta} & 0\\
		-\frac{4}{\beta} & -\frac{12}{\beta} & 0 & 0 & \frac{54}{\beta}
	\end{array}\right),
\end{equation}
\begin{equation}
	R= \frac{1}{3}\left(\begin{array}{ccccc}
		16 & -1 & -1 & -1 & 2 \\
		20 & 25 & -2 & 1 & 1 \\
		4 & 5 & 2 & -1 & -1\\
		4 & 5 & -1 & 2 & -1\\
		4 & 5 & -1 & -1 & 2
	\end{array}\right),
\end{equation} and
\begin{equation}
	M=L^{-1}R
	= \frac{1}{162}\left(\begin{array}{ccccc}
		144 & -9 & -9 & -9 & 18 \\
		8 & 157 & -5 & 13 & -8 \\
		\frac{64}{\beta} & \frac{122}{\beta} & \frac{122}{\beta} & -\frac{58}{\beta} & -\frac{64}{\beta} \\
		\frac{56}{\beta} & -\frac{35}{\beta} & -\frac{35}{\beta} & \frac{91}{\beta} & -\frac{56}{\beta} \\
		-\frac{88}{\beta} & -\frac{26}{\beta} & -\frac{26}{\beta} & -\frac{62}{\beta} & \frac{88}{\beta}
	\end{array}\right).
\end{equation}
When $\beta=1, \sigma(M)>1$, this case correspond to the divergent example in the paper \cite{Caihua}, which is the direct extension of multi-block ADMM.
On the other hand, when $\beta>1$ such that $ \sigma(M)<1$, the iteration Eq.~\eqref{eq:ite} generates convergent sequence $\left\{ (x_1^{(k)}, x_2^{(k)}, x_3^{(k)}, w^{(k)})\right\}$.
The sequence $\left\{ (x_1^{(k)}, x_2^{(k)}, x_3^{(k)})\right\}$
will converge to the unique solution $\bm{0}$,
since
\begin{equation}
    \sigma(M)<1 \Longrightarrow \lim_{k\rightarrow \infty} M^{k} = \bm{0},
\end{equation}
and then
\begin{equation}
	\lim_{k\rightarrow \infty}
	\left(\begin{array}{c}
		x_2^{(k)}\\
		x_3^{(k)}\\
		w^{(k)}
	\end{array}\right)
	= \lim_{k\rightarrow \infty}
	M^{k}
	\left(\begin{array}{c}
		x_2^{(0)}\\
		x_3^{(0)}\\
		w^{(0)}
	\end{array}\right) =\bm{0},
\end{equation}
\begin{equation}
	\lim_{k\rightarrow \infty} x_1^{(k+1)}
	= \lim_{k\rightarrow \infty}
	\left[ -\frac{4}{3}x_2^{(k)} -\frac{5}{3}x_3^{(k)}+ \frac{w_1^{(k)}+w_2^{(k)}+w_3^{(k)}}{3 \rho^{(k)}}
	\right]
	=0.
\end{equation}
For example, let $\beta = 1.1$,  it can be calculated that $\sigma(M)=|0.9398 + 0.2812i|=0.9809<1$, while when $\beta = 1$,  we have $\sigma(M)=|0.9836 + 0.2984i| =1.0278>1$.

\paragraph{Parallelization and scalability.}
Although Algorithm 1 is presented using sequential (Gauss-Seidel-type) block updates, the method is naturally parallelizable. 
In a Jacobi-type implementation \cite{cohen2025alternating}, all blocks $\bm{u}_j^{(k+1)}$ 
can be updated simultaneously using the previous iterate 
$\bm{u}_j^{(k)}$ , followed by the multiplier and penalty updates:
\begin{equation}
\begin{aligned}
   \bm{u}_1^{(k+1)} := & \arg\min_{\bm{u}_1} L_{\rho^{(k)}}(\bm{u}_1,\bm{u}_2^{(k)}, \dots,\bm{u}_J^{(k)},w^{(k)}),  \\
   \bm{u}_j^{(k+1)} := & \arg\min_{\bm{u}_j} L_{\rho^{(k)}}  (\bm{u}_{i<j}^{(k)},\bm{u}_j,\bm{u}_{i>j}^{(k)},w^{(k)}),  \quad j=2,\dots,J-1, \\
	\bm{u}_J^{(k+1)} := & \arg\min_{\bm{u}_J} L_{\rho^{(k)}}(\bm{u}_1^{(k)},\dots,\bm{u}_{J-1}^{(k)}, \bm{u}_J,w^{(k)}),   \\
	w^{(k+1)} := & w^{(k)} + \rho^{(k)} (\sum_{j=1}^J A_j \bm{u}_j^{(k+1)}-\bm{ c }),  \\
    \rho ^{(k+1)} := & \beta \rho^{(k)}, \\
    k:= & k+1.   
\end{aligned}
\end{equation}
The convergence analysis can be adapted to this parallel variant with minor modifications to the descent inequality of the augmented Lagrangian.
The difference in the augmented Lagrangian function becomes
	\begin{equation}\label{eq:P4}
		\begin{aligned}
			& L_{\rho^{(k)}}(\bm{u}_1^{(k)},\dots,\bm{u}_{J}^{(k)},w^{(k)}) - L_{\rho^{(k+1)}}(\bm{u}_1^{(k+1)},\dots,\bm{u}_{J}^{(k+1)},w^{(k+1)}) \\
			=&\sum_{j=1}^{J} \left[f_j(\bm{u}_j^{(k)} )-f_j(\bm{u}_j^{(k+1)})-\langle \nabla f_j(\bm{u}_j^{(k+1)}),\bm{u}_j^{(k)}-\bm{u}_j^{(k+1)} \rangle \right]  \\
			&+ 
            \sum_{j=1}^{J}  \rho^{(k)}  \| A_j \bm{u}_j^{(k+1)} -A_j \bm{u}_j^{(k)} \|^2
            -  \frac{\rho^{(k)}}{2} \|\sum_{j=1}^{J}   (A_j\bm{u}_j^{(k+1)} -  A_j\bm{u}_j^{(k)}) \|^2  \\
            &
			-\frac{\beta + 1}{2}\frac{1}{\rho^{(k)}} \| w^{(k+1)} - w^{(k)} \|^2 \\
			  \geq 
            &\sum_{j=1}^{J} \left[f_j(\bm{u}_j^{(k)} )-f_j(\bm{u}_j^{(k+1)})-\langle \nabla f_j(\bm{u}_j^{(k+1)}),\bm{u}_j^{(k)}-\bm{u}_j^{(k+1)} \rangle \right]  \\
			&+ \sum_{j=1}^{J} \frac{\rho^{(k)}}{2} \| A_j \bm{u}_j^{(k+1)} -A_j \bm{u}_j^{(k)} \|^2
			-\frac{\beta + 1}{2}\frac{1}{\rho^{(k)}} \| w^{(k+1)} - w^{(k)} \|^2 .
		\end{aligned}
	\end{equation} 

Sequential updates typically require fewer outer iterations because they exploit the most recent block information. In contrast, Jacobi-type updates may require more outer iterations due to the use of stale information, but they enable parallel computation across blocks and can reduce wall-clock time when parallel resources are available. 
Therefore, while the present implementation uses sequential updates for simplicity and efficiency at moderate values of $J$, the proposed framework scales naturally to larger $J$ via parallel block updates.

\section{Uniqueness of the stationary point}
\label{sec:unique}

In this section, by introducing  sufficient conditions for quasi-convex functions, we study the uniqueness of the stationary point of our constrained minimization problem in Eq.~\eqref{eq:cp}.

\subsection{Quasi-convex function and unique stationary point}
\label{sec: unique stationary point}
\begin{definition}[Quasi-convex function \cite{boyd2004convex}]
A function $h: \mathbb{R}^n \rightarrow \mathbb{R}$ is called quasi-convex (or unimodal) if its domain $\dom (h)$ and all
its sublevel sets
\begin{equation}
     	S_{\alpha} = \left\{\bm{x} \in \dom(h) \bigg\vert h(\bm{x}) \leq \alpha\right\},
\end{equation}
for $\alpha \in \mathbb{R}$, are convex.
Alternatively, a function $h$ is quasi-convex if and only if its domain $\dom (h)$
is convex, and for any $\bm{x},\bm{y} \in \dom (h)$ and
$ \lambda \in [0,1]$, we have
\begin{equation}
    h(\lambda \bm{x}+(1-\lambda )\bm{y}) \leq \max \left \{h(\bm{x}),h(\bm{y})  \right\}.
\end{equation}
\end{definition}

\begin{theorem}[Unique stationary point]
\label{th:sufficient}
	Let $h: U \rightarrow \mathbb{R}$ be a twice differentiable function on an open convex set $U \subset \mathbb{R}^n$ such that for all $\bm{x} \in U$,

	(S1)
	$\|\nabla h(\bm{x}) \|^2 + p \lambda_{\min} >0$,
	 where $0<p \ll 1$ is a small positive parameter to relax the condition, and
  $\lambda_{\min}$
	is the smallest eigenvalue of the hessian matrix $\nabla^2 h(\bm{x})$;

	(S2)
	$\det(B_k ) \leq 0$, $k=1,2,\dots,n$, where
	$B_k $ is the $(k+1)$-th order leading principal submatrix of $B(\bm{x})$, where
	$$B(\bm{x}) = \left[\begin{array}{cc}
		0 &  \nabla h(\bm{x})^\T  \\
		\nabla h(\bm{x})  &  \nabla ^2 h(\bm{x})
	\end{array}\right]
	$$ is the bordered Hessian matrix when $ \nabla h(\bm{x}) \neq \bm{0}$,
	$B_1 = \left[\begin{array}{cc}
		0 &  h_1  \\
		h_1  &  h_{11}
	\end{array}\right],
	B_2 = \left[\begin{array}{ccc}
		0 &  h_1 & h_2  \\
		h1 & h_{11} & h_{12} \\
		h_2 & h_{21} & h_{22}
	\end{array}\right],
	\dots,
	B_n = B(\bm{x})$,
       and $ h(\bm{x}) =(h_1, \dots, h_n), \nabla^2 h(\bm{x}) = (h_{ij})$.

Then $h(\bm{x})$ is quasi-convex on $U$,
and $h(\bm{x})$ has at most one stationary point on $U$.
\end{theorem}

\begin{proof}
If $h(\bm{x})$ satisfies (S1),
when $\nabla h(\bm{x})=\mathbf 0$,
we have $\lambda_{\min}(\nabla^2 h(\bm{x})) >0$,
which implies $  \nabla^2 h(\bm{x})$ is positive definite.
(S1) is a sufficient condition for the condition (H1) in Theorem~\ref{th:sq} from Ref.~\cite{CJ}.
When $ \nabla h(\bm{x})\neq \bm{0}$,
letting $A=\nabla^2 h(\bm{x}), \bm{a}=\nabla h(\bm{x})$,
(S2) is equivalent to (H2) in Theorem~\ref{th:sq}
by Theorem~\ref{th:equi} \cite{CrouzeixF82}.
Therefore, $h(\bm{x})$ is quasi-convex on $U$ by Theorem~\ref{th:sq}.
	
Suppose that $\bm{x}_1, \bm{x}_2 \in U$ are two stationary points of  $h(\bm{x})$,
i.e., $\nabla h(\bm{x}_1) = \nabla h(\bm{x}_2)= \bm{0}$.
We have that $  \nabla^2 h(\bm{x}_1)$ and $\nabla^2 h(\bm{x}_2) $ are positive definite, and thus both $\bm{x}_1$ and $\bm{x}_2$ are local minima.
Since $h(\bm{x})$ is quasi-convex on $U$,
any point on the line segment should have the same function value:
\begin{equation}
    h(\bm{x}_1)= h(\bm{x}_2)=h(\bm{z}),
\quad
\forall \bm{z} \in \left\{\bm{z} : \bm{z}=\lambda\bm{x}_1+(1-\lambda)\bm{x}_2, 0 \leq \lambda \leq 1 \right\};
\end{equation}
otherwise, the sublevel set cannot be convex.
However, the fact that $h(\bm{u}) $ is a constant on the line segment is in
contradiction with the conclusions that
$\nabla^2 h(\bm{x}_1)$ and $\nabla^2 h(\bm{x}_2)$ are positive definite
obtained above based on (S1).
Hence $\bm{x}_1 = \bm{x}_2$, i.e., the stationary point of $h(\bm{x})$ within domain $U$ is unique if exists.
\hfill\qed
\end{proof}

\begin{remark}
Conditions (S1) and (S2) in Theorem \ref{th:sufficient} are sufficient conditions for quasi-convexity, which also guarantees the uniqueness of the stationary point. We introduce these conditions for our minimization problem based on the available sufficient conditions for quasi-convexity (Theorems A.1 and A.2 in Appendix A, from [11,10]). Note that it is not easy to check the conditions in the definition of quasi-convexity directly. Here we use condition (S1) in Theorem~\ref{th:sufficient} instead of condition (H1) in Theorem~\ref{th:sq}~\cite{CJ}, because (H1) requires positive definiteness of the Hessian matrix at the stationary point, which is unknown before solving our minimization problem.  
\end{remark}

\subsection{Uniqueness for the case of three Burgers vectors}\label{sec:4.2}

Based on the uniqueness of the stationary point of a quasi-convex function given in Theorem~\ref{th:sufficient} in the previous subsection, we establish the uniqueness of the solution of our constrained minimization problem in Eq.~\eqref{eq:cp} for the case of three Burgers vectors.

We consider the case of three Burgers vectors, i.e., $J=3$. An example is the twist grain boundary in fcc crystals \cite{table1,table2,ZHANG}.
The constrained minimization problem in Eq.~\eqref{eq:cp}  in this case becomes
\begin{equation}\label{eq:three}
	\begin{aligned}
		\min_{\bm{u}_1,\bm{u}_2,\bm{u}_3} &  \sum_{j =1}^{3} f_j( \bm{u}_j)  \\
		\text{s.t. }&
		A_1 \bm{u}_1 + A_2\bm{u}_2 + A_3 \bm{u}_3 =\bm{ c }.
	\end{aligned}
\end{equation}

Assume that  $\rank (\bm{b}^{(1)}, \bm{b}^{(2)}, \bm{b}^{(3)})=2$ and
$\rank \left[A_1, A_2, A_3\right]=\rank \left[A_1, A_2, A_3, \bm{c}\right] $ in Eq.~\eqref{eq:three}. In this case, the linear constraints  in Eq.~\eqref{eq:three} have infinitely many solutions, and
we can solve the linear constraints to express $\bm{u}_2, \bm{u}_3$ in terms of $\bm{u}_1$.
We reformulate the objective function in  Eq.~\eqref{eq:three}  as
\begin{equation}
 F(u_{1x},u_{1y}) :=  \sum_{j =1}^{3} f_j( \bm{u}_j(\bm{u}_1) ),
\end{equation}
where $\bm{u}_2(\cdot), \bm{u}_3(\cdot) $ are functions of $\bm{u}_1=(u_{1x}, u_{1y})$, determined by the linear constraints in Eq.~\eqref{eq:three} .
Equivalently, the constrained minimization in Eq.~\eqref{eq:three} is converted to an unconstrained minimization problem defined as
\begin{equation}\label{eq:threeun}
 \min_{u_{1x},u_{1y}} F(u_{1x},u_{1y}),
\end{equation}
where the objective function
$F $ has two variables $u_{1x},u_{1y}$.

We apply Theorem~\ref{th:sufficient} to the objective function in Eq.~\eqref{eq:threeun}
to obtain the uniqueness of the solution of the unconstrained minimization problem in Eq.~\eqref{eq:threeun}, which is equivalent to the constrained minimization problem in Eq.~\eqref{eq:three} for determining the grain boundary dislocation structure with $J=3$ Burgers vectors.

We first have the following theorem by applying Theorem~\ref{th:sufficient} to the objective function
 in the constrained minimination problem in Eq.~\eqref{eq:three} for the case of three Burgers vectors.

\begin{theorem}
[Uniqueness for the three Burgers vectors case]
\label{pro:sufficient J3}
In the constrained minimization problem Eq.~\eqref{eq:three}, assuming that
$\rank (\bm{b}^{(1)}, \bm{b}^{(2)}, \bm{b}^{(3)})=2$ and
$\rank \left[A_1, A_2, A_3\right]=\rank \left[A_1, A_2, A_3, \bm{c}\right] $,
define functions $\bm{u}_2(\cdot), \bm{u}_3(\cdot):\mathbb{R}^{2} \rightarrow \mathbb{R}^{2} $
satisfying the linear constraint
$A_1 \bm{u}_1 + A_2\bm{u}_2(\bm{u}_1) + A_3 \bm{u}_3(\bm{u}_1) =\bm{ c }$.
Let
\begin{equation}
    F(u_{1x},u_{1y}) :=  \sum_{j =1}^{3} f_j( \bm{u}_j(\bm{u}_1)),\label{eqn:u-1}
\end{equation}
where $\bm{u}_1=(u_{1x}, u_{1y})$ and
\begin{equation}
       f_j(\bm{u}_j) =      \left[1-\nu \frac{(\bm{u}_j \times \bm{n} \cdot \bm{b}^{(j)})^2 }{b^2 (\|\bm{u}_j \|^2+\epsilon)}\right]
       b\sqrt{\|\bm{u}_j \|^2 +\epsilon  }    \log \frac{1}{r_g \sqrt{\|\bm{u}_j\|^2 +\epsilon  }} . \label{eqn:u-2}
\end{equation}
Given an open convex set $U \subset \mathbb{R}^2$ and a positive constant $0<p \ll 1$, if
for all $\bm{u}_1 \in U $, we have

       (1) $\|\nabla F\|^2 + p \lambda_{\min}(\nabla^2 F) >0$, that is,
       $ F_1^2+F_2^2+ p \frac{F_{11}+F_{22}-\sqrt{(F_{11}-F_{22})^2+4F_{12}^2}}{2} >0$,
       which is the sufficient condition of the statement that
       $\nabla F( \bm{u}_1)= \bm{0}$
       implies that $  \nabla^2 F(\bm{u}_1)$ is positive definite;

       (2)
       $\left|\begin{array}{cc}
              0 &  F_1  \\
              F_1  &  F_{11}
       \end{array}\right| \leq 0,
       \left|\begin{array}{ccc}
              0 &  F_1  & F_2\\
              F_1  &  F_{11} & F_{12} \\
              F_2 & F_{21} & F_{22}
       \end{array}\right| \leq 0$,\\
       where
       $$(F_1, F_2) = \left(\frac{\partial F}{\partial u_{1x}}, \frac{\partial F}{\partial u_{1y}}\right),\ \
     \left( \begin{array}{cc}
          F_{11} & F_{12} \\
              F_{21} & F_{22}
       \end{array} \right)
       =\left( \begin{array}{cc}
          \frac{\partial^2 F}{ \partial^2 u_{1x}} & \frac{\partial^2 F}{\partial u_{1x} \partial u_{1y}} \\
              \frac{\partial^2 F}{\partial u_{1y} \partial u_{1x}} &  \frac{\partial^2 F}{\partial^2 u_{1y} }
       \end{array} \right),$$
then $F(u_{1x},u_{1y})$ has at most one stationary point on $U$.
Therefore, the solution of the constrained minimization problem Eq.~\eqref{eq:three} is unique over $U $ if it exists.
\end{theorem}

\begin{proof}
The conditions (1) and (2) are the corresponding sufficient conditions in Theorem~\ref{th:sufficient} in the case of two variables.
Using Theorem~\ref{th:sufficient},
We can show that $F(u_{1x},u_{1y})$ is quasi-convex and has at most one stationary point over some open convex domain $U \subset \mathbb{R}^2$.
Note that the minimum point must be a stationary point for a differentiable function.
Hence $F(u_{1x},u_{1y})$ has a unique minimum point over the given domain $U$ if it exists.
\qed
\end{proof}

\begin{remark}
Theorem~\ref{th:sufficient} is stated for a general $n$-dimensional function and is therefore dimension-independent. 
Theorem~\ref{pro:sufficient J3} provides a concrete illustration in the $J=3$ case. 
For the general $J$-block problem with $J>3$, consider the constrained minimization problem
\[
\min_{\bm{u}_1,\dots,\bm{u}_J} \ \sum_{j=1}^J f_j(\bm{u}_j)
\quad
\text{subject to}
\quad
\sum_{j=1}^J A_j \bm{u}_j = \bm{c}.
\]
Assume that the constraint satisfies 
\[
\rank[A_1,\dots,A_J] = \rank[A_1,\dots,A_J,\bm{c}],
\]
so that the feasible set is nonempty. 
After eliminating the linear constraints, the problem can be reformulated as an equivalent unconstrained problem
\[
F(\bm{x}) := \sum_{j=1}^J f_j(\bm{u}_j(\bm{x})), 
\quad
\bm{x} =(x_1,\dots, x_{2r}) \in U \subset \mathbb{R}^{2r},
\]
where $2r<2J$ is the number of free variables and $U$ is an open convex set.
Given a positive constant $0<p \ll 1$, if for all $\bm{x} \in U$ the function $F$ satisfies:
\begin{itemize}
\item[(1)] $\|\nabla F(\bm{x}) \|^2 + p\,\lambda_{\min}(\nabla^2 F(\bm{x})) > 0$,
\item[(2)] $\det(B_k(\bm{x})) \le 0$, $k=1,\dots,2r$, where
\[
B_k(\bm{x}) =
\begin{bmatrix}
0 & \frac{\partial F}{\partial x_1} 
& \dots &  \frac{\partial F}{\partial x_k}\\
\frac{\partial F}{\partial x_1}  & 
\frac{\partial^2 F}{\partial^2 x_1  }
&   \dots &
\frac{\partial^2 F}{\partial x_1 \partial x_k }\\
\vdots & \vdots & \ddots &  \vdots \\
\frac{\partial F}{\partial x_k}  & 
\frac{\partial^2 F}{\partial x_k \partial x_1 }
&  \dots
&  \frac{\partial^2 F}{\partial^2 x_k } \\
\end{bmatrix}
\in \mathbb{R}^{(k+1)\times (k+1)},
\]
is the $(k+1)$-th order leading principal submatrix of the bordered Hessian
\[
B(\bm{x}) =
\begin{bmatrix}
0 & \nabla F(\bm{x})^\T \\
\nabla F(\bm{x}) & \nabla^2 F(\bm{x})
\end{bmatrix}
\in \mathbb{R}^{(2r+1)\times (2r+1)},
\]
\end{itemize}
then $F$ is quasi-convex on $U$ and has at most one stationary point. 
Consequently, the original constrained minimization problem admits at most one solution in $U$ (if it exists).
\end{remark}

We provide an example in Proposition~\ref{pro:sufficient J3 spf} below.
Also see Figure~\ref{fig:fsc} for an example of the objective function in Proposition~\ref{pro:sufficient J3 spf}.
The quantities in conditions (S1) and (S2)  in Theorem~\ref{pro:sufficient J3} for this example are shown in Figure~\ref{fig: condition}.

\begin{figure}[hbtp]
	\centering
	\subfigure[$f_1$ surface plot]
				{\includegraphics[width=0.45\textwidth]{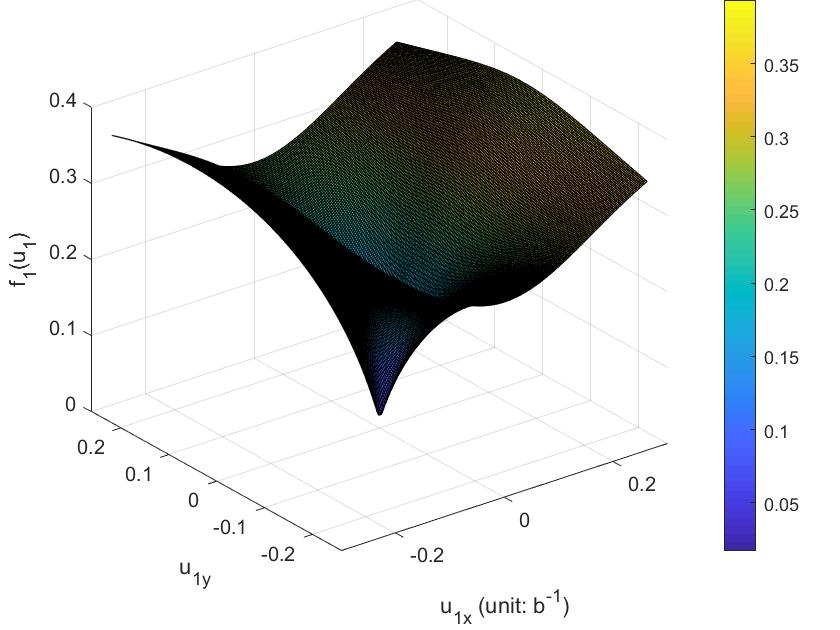}
		\label{fig: f1}
	}
    \subfigure[$f_2$ surface plot]
    {
    	    	\includegraphics[width=0.45\textwidth]{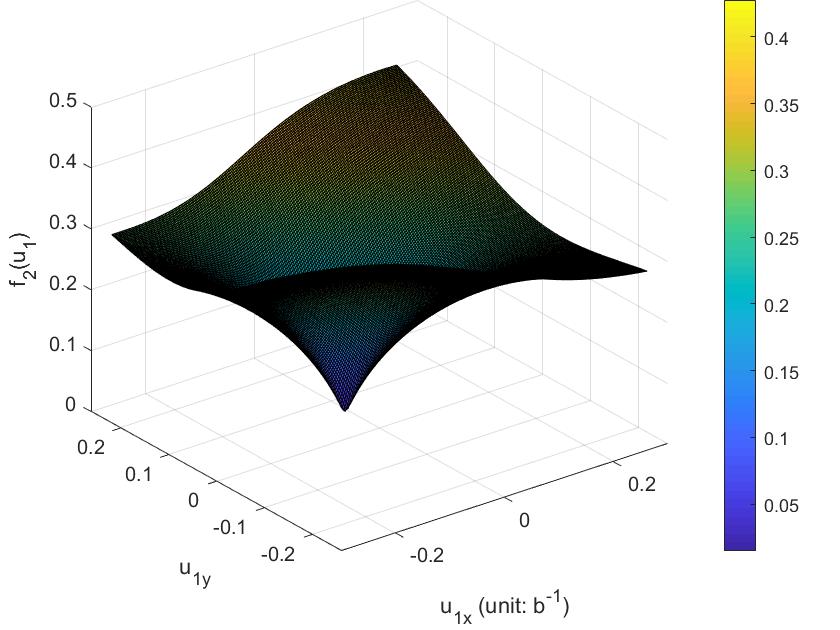}
    	}
    \subfigure[$f_3$ surface plot]{
    	    	\includegraphics[width=0.45\textwidth]{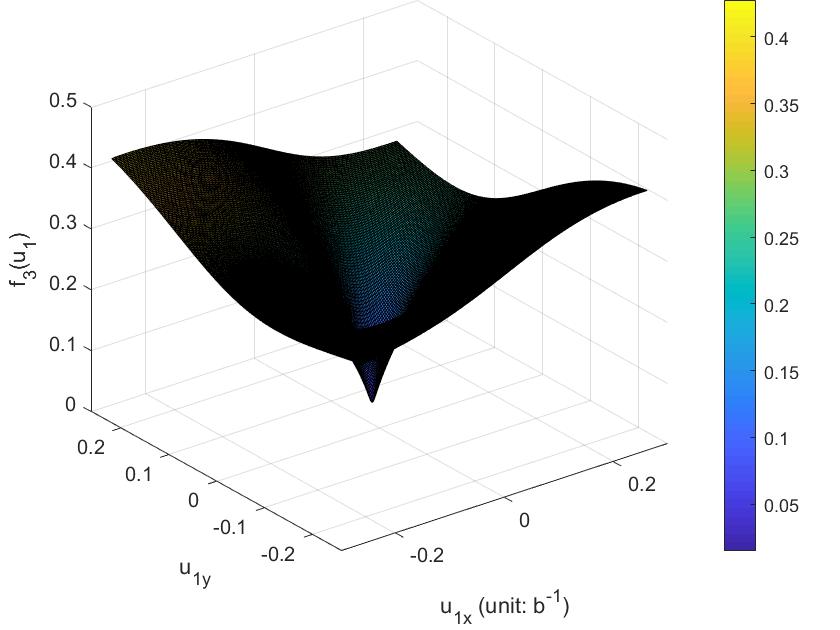}
     }
    \subfigure[Contour plots of $f_1$ (yellow solid), $f_2$ (green dot), $f_3$ (blue dash)]{
    	    	\includegraphics[width=0.45\textwidth]{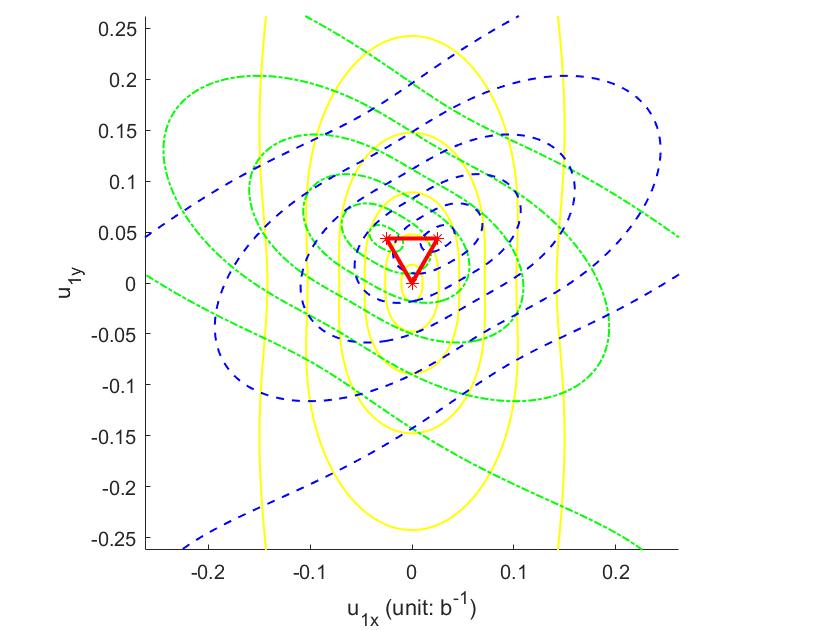}
     }
	\subfigure[$F=f_1+f_2+f_3$ surface plot]{
		\includegraphics[width=0.45\textwidth]{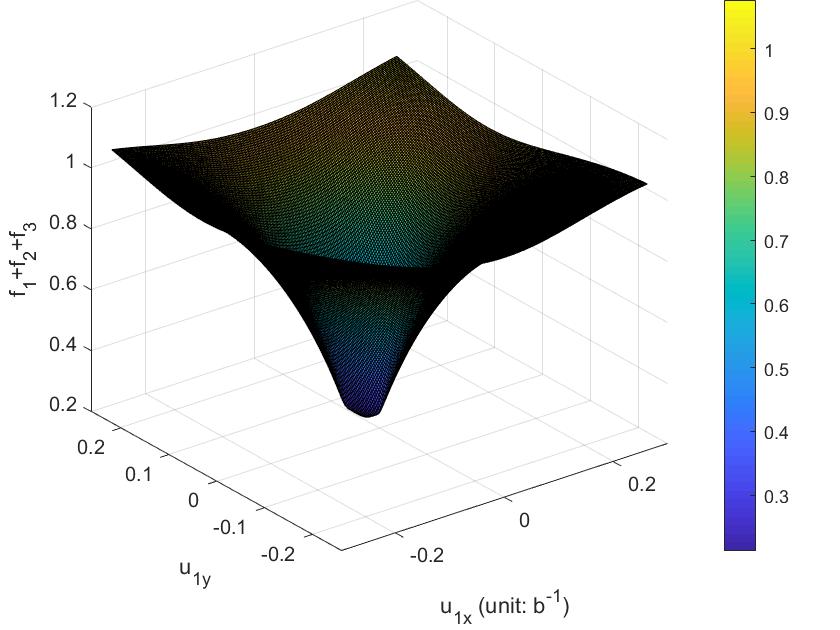}
	}
	\subfigure[$F=f_1+f_2+f_3$ contour plot]{
		\centering
		\includegraphics[width=0.45\textwidth]{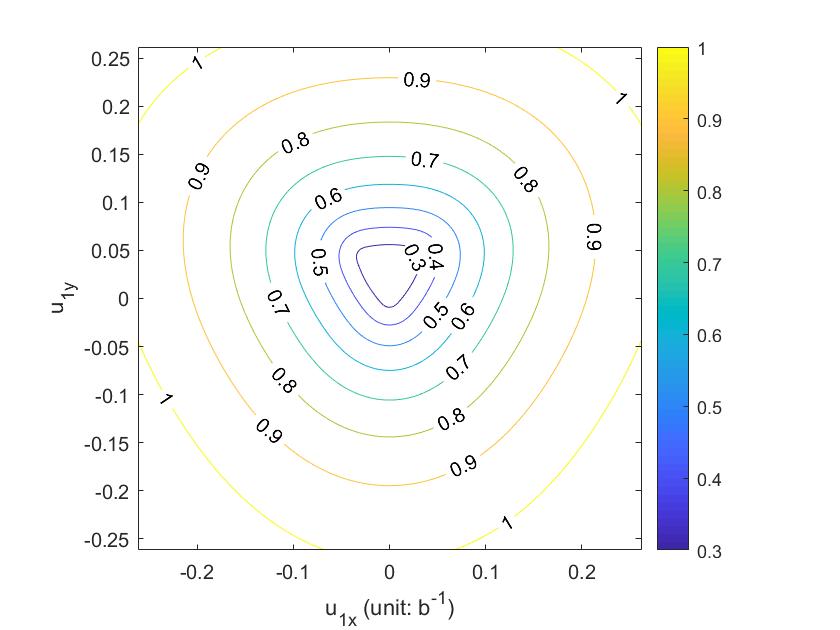}
	}
	\caption{
Surface and contour plots of component functions $f_j$, $j=1,2,3$, and the objective function $F=f_1+f_2+f_3$ in terms of variables $u_{1x}$ and $u_{1y}$ in Proposition~\ref{pro:sufficient J3 spf}. Here $\theta=2.5^\circ$, and the small parameter $\epsilon=\frac{1}{400} (\frac{\theta}{b})^2 $.
}
	\label{fig:fsc}
\end{figure}

\FloatBarrier

\begin{proposition}
[Specific three Burgers vectors case]
\label{pro:sufficient J3 spf}
Consider the three Burgers vectors case with
$
\bm{b}^{(1)} =(1,0,0)^\T b$, $\bm{b}^{(2)} =(\frac{1}{2},\frac{ \sqrt{3}}{2},0)^\T b$, $\bm{b}^{(3)}=(\frac{1}{2}, -\frac{ \sqrt{3}}{2},0)^\T b$,
$ \nu=0.347$, $\bm{a}=(0,0,1)^\T$, $\bm{n}=(0,0,1)^\T$, and $r_g =0.85 b$.
This is a $\{111\}$ planar twist grain boundary in fcc aluminum \cite{table1,table2,ZHANG}.
The constrained minimization problem  Eq.~\eqref{eq:three} in this case is
\begin{equation}
       \begin{aligned}
              \min &  \sum_{j =1}^{3} f_j( \bm{u}_j)  \\
              \text{s.t. }& u_{2x}=-u_{1x}-\frac{\theta}{\sqrt{3}b},
              & u_{3x}=-u_{1x}+\frac{\theta}{\sqrt{3}b},\\
              & u_{2y}=-u_{1y}+\frac{\theta}{b},
              & u_{3y}=-u_{1y}+\frac{\theta}{b},
       \end{aligned}
\end{equation}
There exist $\epsilon_0>0$ that depends on the value of $\theta$,
such that when the parameter $\epsilon \geq \epsilon_0$ in the objective function,
the solution of the equivalent unconstrained minimization problem
\begin{flalign*}
&\min_{\bm{u}_1}  F(u_{1x},u_{1y}), \\
&{\rm where} \ F(u_{1x},u_{1y}) =  \sum_{j =1}^{3} f_j( \bm{u}_j(\bm{u}_1)),
\end{flalign*}
is unique in the domain  $U:=\{ \bm{u}_1 \in \mathbb{R}^2: \|\bm{u}_1 \| \leq \frac{\theta_m}{b}\}$. Here $\theta_m=\frac{\pi}{12}=15^\circ$, which is the maximum possible angle for a low angle grain boundary.
\end{proposition}

\begin{proof}
We verify that $F(u_{1x}, u_{1y})$ satisfies the conditions in Theorem~\ref{pro:sufficient J3} on $U$,
if given an appropriate  positive constant $\epsilon$.
For example, we set $\theta = 2.5^\circ$ and $\epsilon = \frac{1}{400}(\frac{\theta}{b})^2$.
In Figure~\ref{fig: condition},
we plot the surface and zero-level set of the quantities in Theorem~\ref{pro:sufficient J3}.
Figure~\ref{fig:s1}  shows the value of $ F_1^2+F_2^2+ p\frac{F_{11}+F_{22}-\sqrt{(F_{11}-F_{22})^2+4F_{12}^2}}{2}$ with $p=0.01$,
and Figure~\ref{fig:s2}  presents the determinant value $\left|\begin{array}{ccc}
              0 &  F_1  & F_2\\
              F_1  &  F_{11} & F_{12} \\
              F_2 & F_{21} & F_{22}
       \end{array}\right| $.
Therefore, Theorem~\ref{pro:sufficient J3} holds for $F(u_{1x}, u_{1y})$ on $U$.
As a result, $F(u_{1x}, u_{1y})$ is quasi-convex on $U$ and its minimum is unique over the domain $U$, see Figure~\ref{fig:fsc}(e) and (f).
\hfill\qed
\end{proof}

\FloatBarrier
\begin{figure}[!hbtp]
	\centering
	\subfigure[(S1) surface plot]
		{\includegraphics[width=0.45\textwidth]{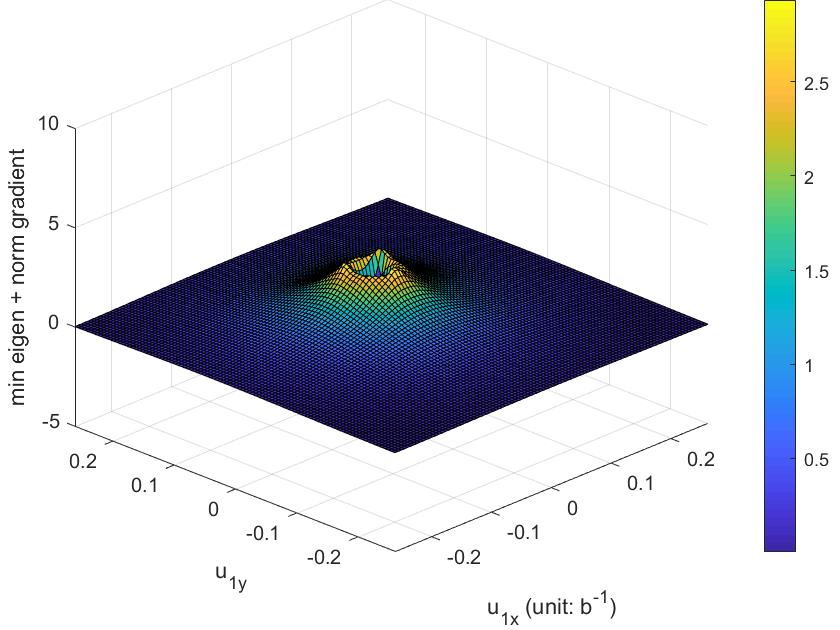}
		          \label{fig:s1}
	}
	\subfigure[(S1) zero-level set]{
		\includegraphics[width=0.45\textwidth]{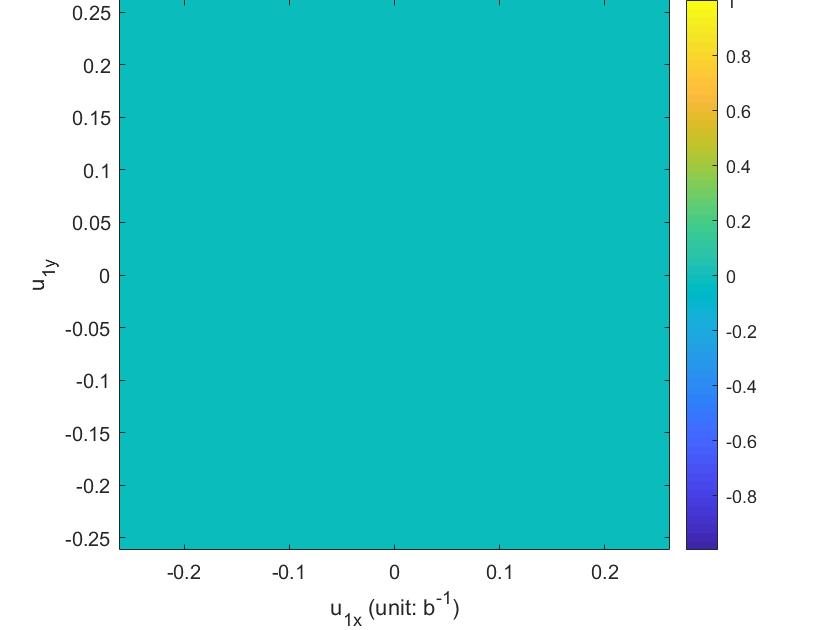}
		\label{fig:z1}
	}	
	\subfigure[(S2) surface plot]{
		\includegraphics[width=0.45\textwidth]{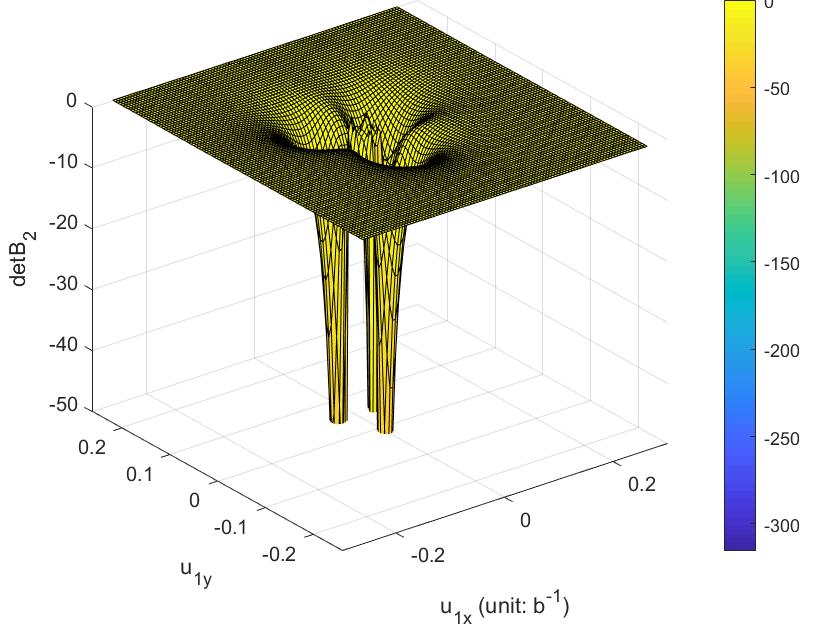}
           \label{fig:s2}
	}
	\subfigure[(S2) zero-level set]{
		\includegraphics[width=0.45\textwidth]{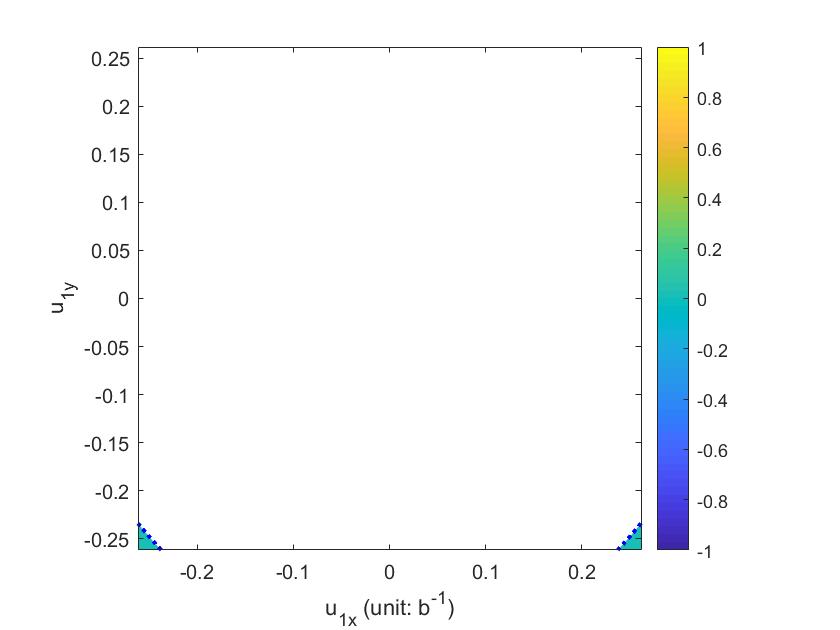}
        \label{fig:z2}
	}
	\caption{Surface and zero-level set plots of the quantities in Proposition~\ref{pro:sufficient J3 spf}  in terms of $u_{1x}, u_{1y}$.  The regions in green in (b) and (d) represent the domain on which the function is positive. Here (S1) denotes the function in the condition (S1), and (S2) denotes the determinant of the $3\times 3$ matrix in the condition (S2).}
	\label{fig: condition}
\end{figure}

\begin{remark} Note that in Proposition~\ref{pro:sufficient J3 spf}, the domain
 $U:=\{ \bm{u}_1 \in \mathbb{R}^2: \|\bm{u}_1 \| \leq \frac{\theta_m}{b}\}$ is the physically-meaningful domain for a low angle grain boundary.
The condition
$\|\bm{u}_1 \| \leq \frac{\theta_m}{b}$, where $\theta_m=\frac{\pi}{12}=15^\circ$, comes from the fact that the misorientation angle $\theta$ for a low angle grain boundary does not exceed  $15^\circ$ \cite{table1,table2}. If this condition does not hold, i.e., $\|\bm{u}_1 \|> \frac{\theta_m}{b}$, the inter-distance between the $\mathbf b_1$-dislocations will be less than $\frac{b}{\theta_m}=3.82b$. In this case, the core regions of these dislocations will  overlap, and there will be no clear dislocation structure of the grain boundary.
\end{remark}

\begin{remark}\label{rmk3}
Following Proposition~\ref{pro:sufficient J3 spf}, in practice, we can choose the small regularization parameter $\epsilon$ in the objective function (Eqs.~\eqref{eqn:u-1} and \eqref{eqn:u-2})  to be $\epsilon_0$, which is the smallest value for the uniqueness of the minimum to hold.
In Table~\ref{tab:e1}, we list the value of $\epsilon_0$ such that the conditions hold for the objective function to have a unique minimum in the practical feasible region, for different misorientation
angle $\theta=2.5^\circ, 3.75^\circ, 7.5^\circ$,  respectively. Since the order of $\mathbf u_j$ (or $\nabla \eta_j$) is $\frac{\theta}{b}$, the errors induced by these values of regularization parameter $\epsilon$ are indeed small.  
Physically, this regularization parameter $\epsilon_0$, which provides a lower cut-off value for $\|\nabla \eta_j\|$, sets an upper cut-off length scale that is consistent with the length scales of the dislocation structures being considered in experimental observations and atomistic simulations. (Recall that the inter-dislocation distance is $1/\|\nabla \eta_j\|$.) For a dislocation array with very large inter-dislocation distance, the stability of the dislocation structure is weak \cite{table2,table1,Zhu2011}, and other physical mechanisms will be dominant. This means that  the uniqueness conclusion in Proposition 2 holds for physically meaningful parameters range of the low angle grain boundaries.

\begin{table}[htbp]
 \centering
       \caption{Value of $\epsilon_0$ for different misorientation
              angle $\theta$.  }
       \label{tab:e1}
              \begin{tabular}{cccc}
                     \hline\noalign{\smallskip}
                     $\theta$ ($^\circ$) &  2.5 & 3.75 & 7.5 \\
                \noalign{\smallskip}\hline\noalign{\smallskip}
                     $\epsilon_0$ $\left((\frac{\theta}{b})^2\right)$& $\frac{1}{400}$ & $\frac{1}{250}$ & $\frac{1}{92}$  \\
                     \noalign{\smallskip}\hline
              \end{tabular}
\end{table}
\end{remark}


Finally, we remark that if both the ADMM limit $ \bm{u}^{\star}$ and the minimum of the objective function are in domain $U$,
then $\bm{u}^{\star}$ is the minimum of the constrained minimization problem Eq.~\eqref{eq:cp}.
Because  for a differentiable function, the minimum point should be a stationary point, and quasi-convexity indicates that the objective function at most have one stationary point on the subdomain that contains all the possible minima in $U$.

\section{Numerical simulations}
\label{sec: numerical}
We use our modified ADMM algorithm to solve the constrained minimization problem in Eq.~\eqref{eq:cp},
which is
$$ 
	\begin{aligned}
		\min  &  \sum_{j =1}^{J} f_j( \bm{u}_j)  \\
		\text{s.t. }&  \sum_{j =1}^{J} A_j \bm{u}_j =\bm{ c }.
	\end{aligned}
$$

In the simulations, we consider a $(1 1 1)$ twist boundary in aluminum which has the fcc lattice structure \cite{table2}.
The crystallography directions $[\bar{1} 1 0],[ \bar{1 1} 2]$ and $[1 1 1]$ are chosen to be the $x, y,z$ axes, respectively. In this system, there are $J=6$  Burgers vectors $ \bm{b}^{(1)}=b(1,0,0)^\T, \:
\bm{b}^{(2)}=b(\frac{1}{2},\frac{\sqrt{3}}{2},0)^\T, \:
\bm{b}^{(3)}=b(\frac{1}{2},-\frac{\sqrt{3}}{2},0)^\T, \:
\bm{b}^{(4)}=b(0,\frac{\sqrt{3}}{3},\frac{\sqrt{6}}{3})^\T, \:
\bm{b}^{(5)}=b(\frac{1}{2},\frac{\sqrt{3}}{6},-\frac{\sqrt{6}}{3})^\T, \:
\bm{b}^{(6)}=b(-\frac{1}{2},\frac{\sqrt{3}}{6},-\frac{\sqrt{6}}{3})^\T$, with the same length $b$, which is $0.286 \mathsf{nm}$.
The grain boundary unit normal vector is $\bm{n}=(0,0,1)^\T$ and the unit vector of the rotation axis is
$\bm{a}=(0,0,1)^\T $ in the original formulation in Eqs.~\eqref{eq:of ori}-\eqref{eq:frank}. The Poisson ratio $\nu=0.347$.
The values of parameters are   $r_g =0.85b$,  $\theta=2.5^\circ$, and $\epsilon=\frac{1}{325}(\frac{\theta}{b})^2$.

In our ADMM algorithm, the initial penalty parameter $ \rho^{(0)}=100$, and the increasing penalty factor $\beta=1.001 $. The  time step  the gradient descent for each block is $\alpha_{\text{ADMM}}=5\times 10^{-4}$. The initial Lagrangian multiplier is
$ w^{(0)}= \bm{0} \in \mathbb{R}^{6}$, and the initial point $\bm{u}^{(0)}=\bm{0} \in \mathbb{R}^{12}$. 
The ADMM iteration is terminated when both the stationarity condition
$\|\nabla_{u} L_{\rho}( \bm{u}_1,\dots,\bm{u}_J, w)\| \leq 10^{-8}$ 
and the primal feasibility condition
$\| \sum_{j=1}^{J}  A_j\bm{u}_j^{(k)}  -\bm{c} \|\leq 10^{-8}  $ 
(in the unit $1/b$).
The subproblems are solved using the gradient descent method, with the following termination conditions: (1) $\|\nabla_{\bm{u}_j} L_{\rho^{(k)}}\|_2<10^{-10}$; or (2) the number of inner iterations reaches 100.
The results are compared with those using the penalty method and the ALM.
For the penalty method, the penalty parameter is $\rho=800$, and the gradient descent time step is $\alpha_{\text{PM}}=5\times10^{-4}$;
for ALM,  $\rho=100$, $\alpha_{\text{ALM}}=5\times10^{-4}$; and other parameters in these two methods are the same as those in the ADMM. Note that  convergence of the penalty method and the ALM applied to the constrained minimization problem  in Eq.~\eqref{eq:cp} can be proved by Theorem~\ref{th:p} \cite{Edwin} and Theorem~\ref{th:a} \cite{Bertsekas}; see Appendix~\ref{sec:convergences}.

\begin{figure}[!hbtp]
	\centering
	\subfigure[]{
     \includegraphics[width=0.45\textwidth]{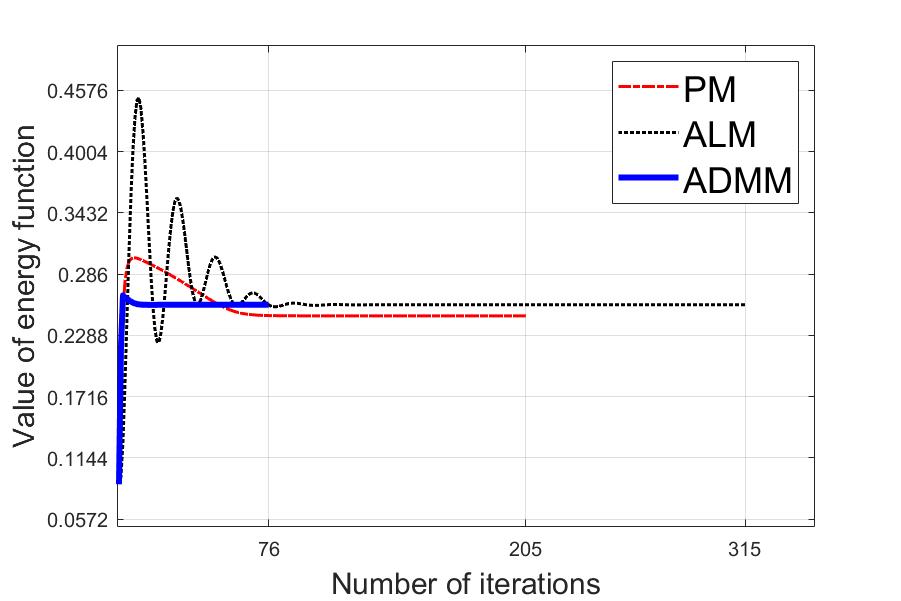}
		\label{fig:f}
	}
    \subfigure[]{
    	\includegraphics[width=0.45\textwidth]{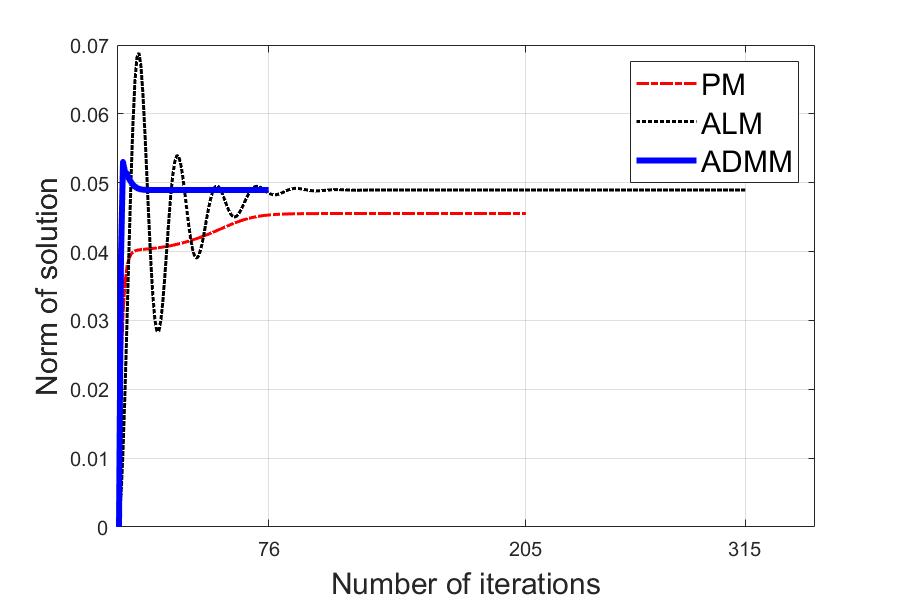}
    	\label{fig:u}
    }
	\subfigure[]{
		\includegraphics[width=0.45\textwidth]{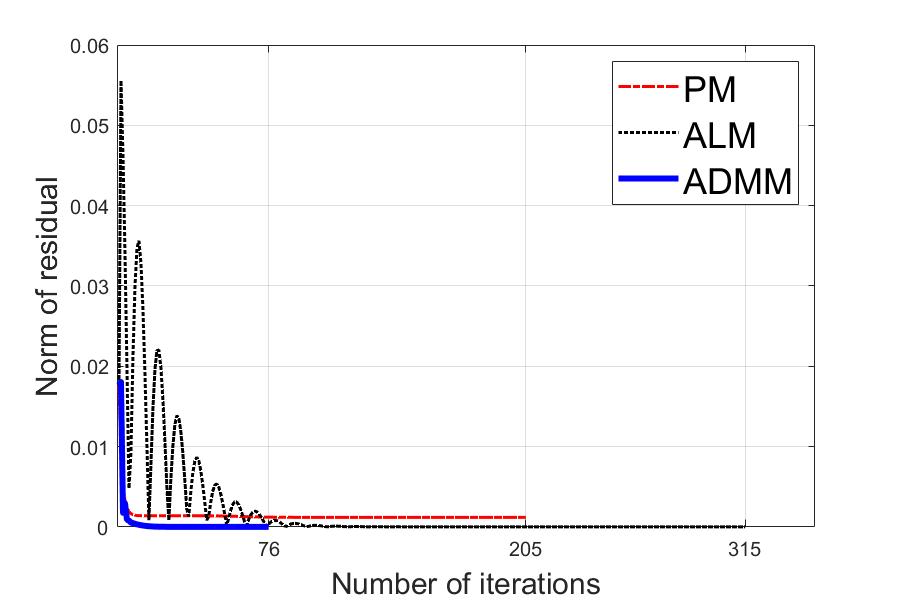}
		\label{fig:r}
	}
	\subfigure[]{
  \includegraphics[width=0.45\textwidth]{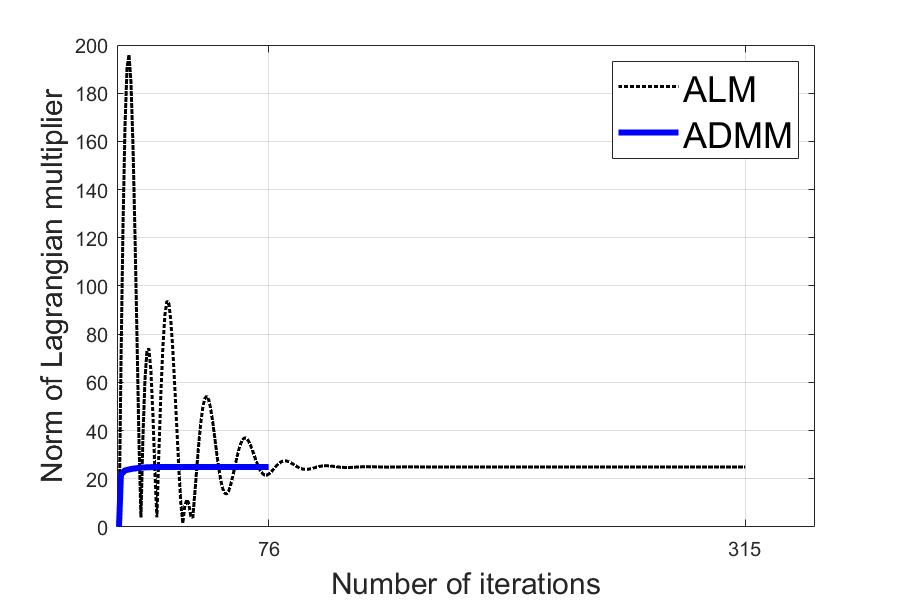}
		\label{fig:w}
	}
	\caption{Simulation results obtained by using our ADMM and comparisons with those by using the penalty method (PM) and the ALM.
(a) Value of the objective function $\sum_{j=1}^6 f_j(\bm{u}_j)$.
		(b) The norm of the solution $\bm{u} $.
            (c) The norm of the residual $ \|\sum_{j=1}^J A\bm{u}_j- \bm{c}\|$, which measures how the constraints are satisfied.
		(d) The norm of the Lagrangian multiplier $w$.}
	\label{fig:fig4}
\end{figure}

Simulation results obtained by using our ADMM and comparisons with those by using the penalty method  and the ALM are shown in Figure~\ref{fig:fig4}.
These three numerical methods are examined by the evolutions of the objective function $\sum_{j=1}^6 f_j(\bm{u}_j)$  (Figure~\ref{fig:f}),
the norm of the solution $ \|\bm{u}\|$ (Figure~\ref{fig:u}), and
the norm of the residual of the constraint $ \|\sum_{j=1}^J A\bm{u}_j- \bm{c}\|$ (Figure~\ref{fig:r}),
and the norm of the Lagrangian multiplier $\|w\|$ (Figure~\ref{fig:w})
in the iteration processes.
It can be seen that under the same stopping criterion, our ADMM  stops much  earlier than the penalty method and the ALM; see Figure~\ref{fig:f}.
From Figure~\ref{fig:r}, when the algorithms stop, both the ADMM and the ALM reach the feasible region
$ \{\bm{u}| \sum_{j=1}^J A\bm{u}_j= \bm{c}\}$ accurately, while the penalty method gives a solution with larger residual in the constraint and  a noticeable deviation from the solutions as shown in  Figure~\ref{fig:u}.
In fact, the penalty method requires a much larger penalty parameter $\rho$ to achieve a solution that is as accurate as those given by the ADMM and the ALM.

Hence, compared to the penalty method, our ADMM converges much faster and gives much more accurate solution that satisfies the constraint more accurately.
Compared to the ALM, our ADMM can achieve a solution to the constrained minimization problem with same level of accuracy in much less iterations, and without the oscillations in the ALM.
We also plot the norm of the Lagrangian multiplier $w$ during iteration processes of the ADMM and the ALM; see Figure~\ref{fig:w}.
It can be seen that the Lagrangian multiplier is indeed bounded and converges, which is consistent with the assumption of boundedness of multipliers. 
Moreover, we construct the matrix $P$ according to Proposition~\ref{pro:condition_for_w} and compute its spectral radius directly. The numerical result shows that
\[
\rho(P) \approx 0.6133 < 1,
\]
which is consistent with the stability condition in Proposition~\ref{pro:condition_for_w}. This provides numerical evidence for the boundedness of the multiplier sequence and is in agreement with the convergence behavior observed in the experiments. 

The solution found by ADMM is $(  0,     0.028, -0.024,  0.014,  0.024, 0.014, 0,  0, \\ 0, 0, 0, 0 )b^{-1}$, which agrees with the theoretical result of the  twist boundary \cite{table1,table2}.
In Table~\ref{tab:e2}, we show results of the density of the $\bm{b}^{1}$-dislocations for this twist grain boundary with different values of misorientation angle $\theta$ obtained by using 
our ADMM algorithm, the simulation results obtained in Ref.~\cite{ZHANG} using the penalty method as well as the theoretical values in Ref.~\cite{table1,table2}. It can be seen that
these three results are consistent, and the errors are within the range caused by the regularization parameter $\epsilon$ discussed in Remark \ref{rmk3}.
%

\begin{table}[!htbp]
	\caption{Density of $\bm{b}^{1}$-dislocations on the $[111]$ twist boundaries. Unit: $b^{-1}$.}
	\label{tab:e2}
	\centering
		\begin{tabular}{cccc}
			\hline\noalign{\smallskip}
			$\theta$ ($^\circ$ )&  2.5 & 3.75 & 7.5  \\
			\noalign{\smallskip}\hline\noalign{\smallskip}
			ADMM & 0.028   & 0.042 &  0.082   \\
			\noalign{\smallskip}\hline\noalign{\smallskip}
			Simulation in \cite{ZHANG} & 0.028   & 0.042 &  0.085  \\
			\noalign{\smallskip}\hline\noalign{\smallskip}
			Theoretical value & 0.029 & 0.044 & 0.087   \\
			\noalign{\smallskip}\hline
		\end{tabular}
\end{table}

\paragraph{Sensitivity analysis on $\rho^{(0)}$.} 
For each block subproblem, the Hessian  
$\nabla^2_{\bm{u}_j} L_j(\bm{u}_j) 
	=\nabla^2 f(\bm{u}_j) +\rho b^2 I_2$
is positive definite when $\rho >\frac{C}{b^2} $, where $C >0$ is a positive constant such that
$\| \nabla ^2 f_j \| \leq C$ (as stated in Lemma~\ref{le:f}). 
Thus, choosing an initial penalty parameter $\rho^{(0)} > \frac{C}{b^2}$ guarantees strong convexity of each block subproblem.
Satisfying this condition, we further conduct simulations with $\rho^{(0)}=30,100,300$ in our ADMM algorithm.
Figure~\ref{fig:sensitivity} presents the evolution of four key quantities during the iterations: the objective function $\sum_{j=1}^6 f_j(\bm{u}_j)$  (Figure~\ref{fig:saf}),
the norm of the solution $ \|\bm{u}\|$ (Figure~\ref{fig:sau}), 
the norm of the constraint residual $ \|\sum_{j=1}^J A\bm{u}_j- \bm{c}\|$ (Figure~\ref{fig:sar}),
and the norm of the Lagrangian multiplier $\|w\|$ (Figure~\ref{fig:saw}).
The results reveal that the choice of $\rho^{(0)}$ does influence the convergence speed,
although the solutions under all settings  eventually converge to the same one of the constrained minimization problem in Eq.~\eqref{eq:cp}.
When $\rho^{(0)} = 30$, the algorithm converges relatively slower due to the weaker initial constraint enforcement.
In contrast, a larger value such as $\rho^{(0)} = 300$ enforces the constraints faster in early iterations, but can induce oscillatory behavior and slower reduction of the objective value due to over-penalization.
The intermediate value $\rho^{(0)} = 100$ offers a good balance between stability and convergence speed. 
These findings indicate that, while the final solution is relatively insensitive to the choice of the initial penalty parameter, a well-chosen value does improve the convergence rate and numerical stability of the algorithm.

\begin{figure}[!hbtp]
	\centering
	\subfigure[]{
     \includegraphics[width=0.4\textwidth]{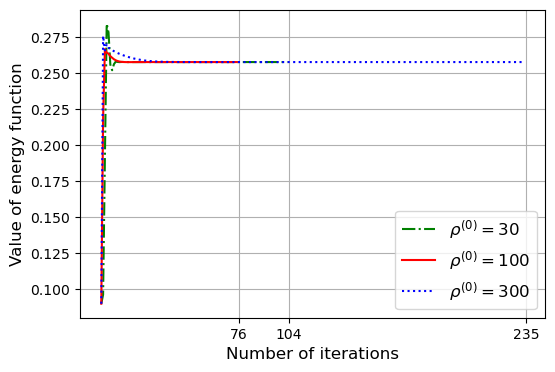}
		\label{fig:saf}
	}
    \subfigure[]{
    	\includegraphics[width=0.4\textwidth]{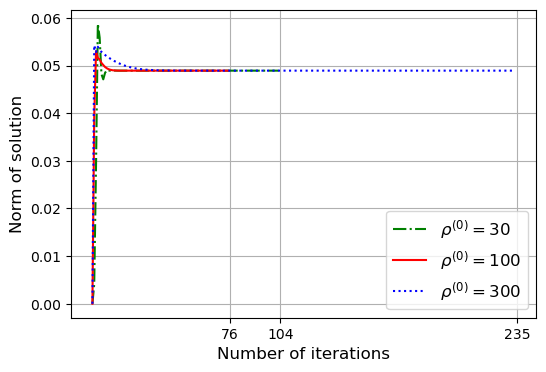}
    	\label{fig:sau}
    }
	\subfigure[]{
		\includegraphics[width=0.4\textwidth]{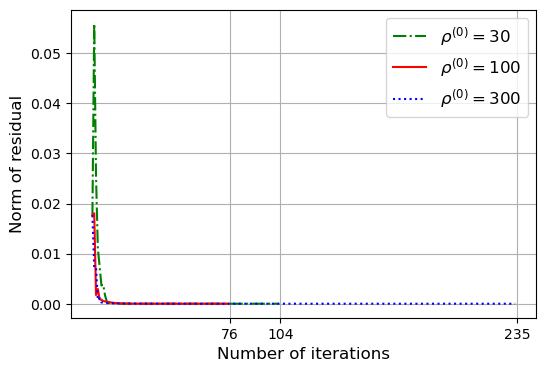}
		\label{fig:sar}
	}
	\subfigure[]{
  \includegraphics[width=0.4\textwidth]{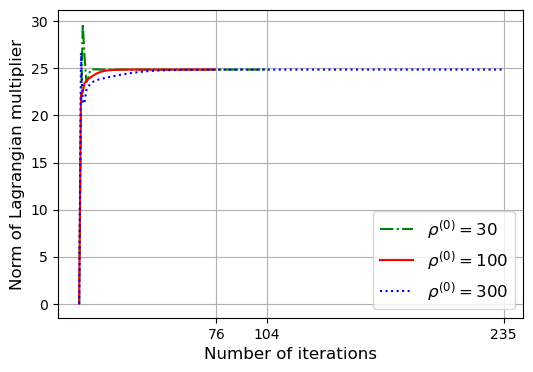}
		\label{fig:saw}
	}
	\caption{ Simulation results obtained by using our ADMM with $\rho^{(0)}=30,100,300$, respectively.
            (a) Value of the objective function $\sum_{j=1}^6 f_j(\bm{u}_j)$.
		(b) The norm of the solution $\bm{u} $.
		(c) The norm of the residual $ \|\sum_{j=1}^J A\bm{u}_j- \bm{c}\|$, which measures how the constraints are satisfied. 
		(d) The norm of the Lagrangian multiplier $w$.}
	\label{fig:sensitivity}
\end{figure}

\FloatBarrier 
\paragraph{Sensitivity analysis on $\beta$.}
To assess the impact of the penalty increment factor $\beta$, we
conducted additional simulations with $\beta=1.0005,1.001,1.05$, and analyzed their influence on the convergence behavior. 
Specifically, we report evolutions of the objective function, the norm of the solution, the constraint residual,
and the Lagrange multiplier norm over the iterations (see Fig. \ref{fig:beta_sensitivity}).
Our results indicate that excessively large values of $\beta$ lead to overly rapid growth of the penalty parameter $\rho^{(k)}$, which can cause numerical instability and deteriorate convergence behavior. For instance, when $\beta=1.1$, the penalty parameter grows exponentially and reaches values beyond the numerical representability, leading to $\rho^{(k)} = \text{NaN}$ undefined. Consequently, the primal updates in the modified ADMM become ill-defined and the algorithm diverges.
On the other hand, choosing $\beta$ too close to 1 results in very slow penalty growth and consequently requires substantially more iterations to meet the same stopping criteria.

\begin{figure}[!hbtp]
	\centering
	\subfigure[]{
     \includegraphics[width=0.4\textwidth]{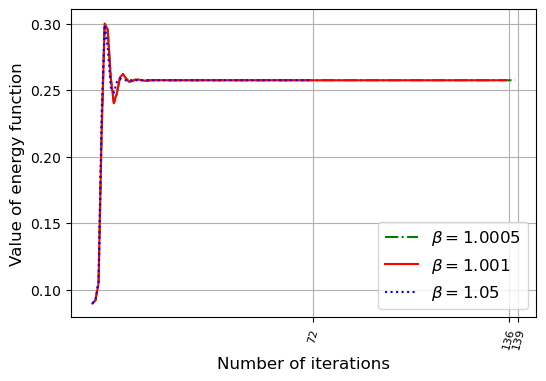}
		\label{fig:betaf}
	}
    \subfigure[]{
    	\includegraphics[width=0.4\textwidth]{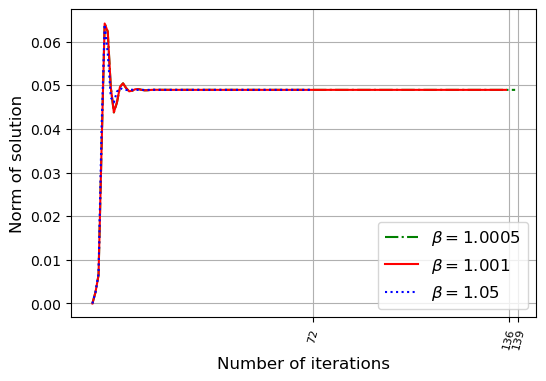}
    	\label{fig:betau}
    }
	\subfigure[]{
		\includegraphics[width=0.4\textwidth]{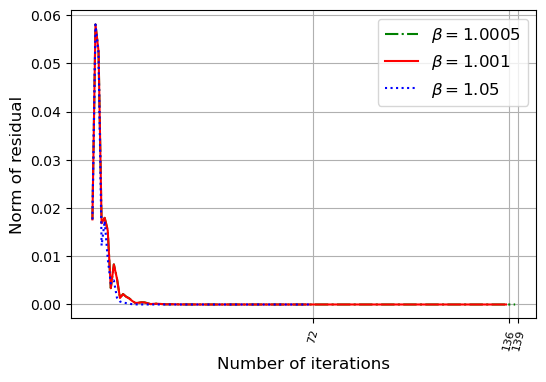}
		\label{fig:betar}
	}
	\subfigure[]{
  \includegraphics[width=0.4\textwidth]{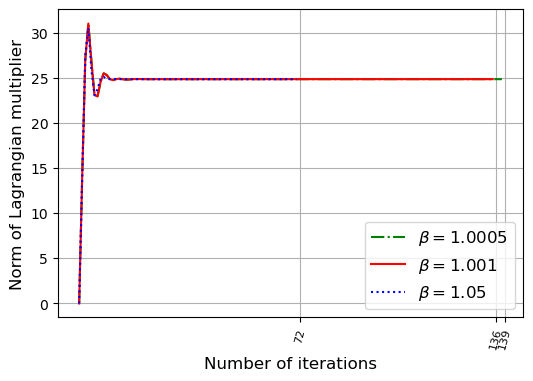}
		\label{fig:betaw}
	}
	\caption{  Simulation results obtained by using our ADMM with  $\rho^{(0)}=20$ and varying $\beta=1.0005,1.001,1.05$, respectively.
            (a) Value of the objective function $\sum_{j=1}^6 f_j(\bm{u}_j)$.
		(b) The norm of the solution $\bm{u} $.
		(c) The norm of the residual $ \|\sum_{j=1}^J A\bm{u}_j- \bm{c}\|$, which measures how the constraints are satisfied. 
		(d) The norm of the Lagrangian multiplier $w$.} 
	\label{fig:beta_sensitivity}
\end{figure}

\paragraph{Numerical verification of uniqueness.} 
We perform a numerical examination of the uniqueness for the above example in the case $J=6$. 
By eliminating dependent variables, the linear constraints in Eq.~\eqref{eq:cp} can be rewritten in terms of the independent variables as 
\begin{equation}
    \begin{cases}
    u_{2x} =-u_{1x}-u_{5x}-\frac{\theta}{\sqrt{3}b},\\
    u_{3x} =-u_{1x}+u_{6x}+\frac{\theta}{\sqrt{3}b},\\
    u_{4x} =u_{5x}+u_{6x},\\
    u_{2y} =-u_{1y}-u_{5y}+\frac{\theta}{b},\\
    u_{3y} =-u_{1y}+u_{6y}+\frac{\theta}{b},\\
    u_{4y} =u_{5y}+u_{6y}.
    \end{cases}       
\end{equation}
This leads to the equivalent unconstrained problem 
\begin{align}\label{eqn:6d}
    \min F(u_{1x},u_{1y},u_{5x},u_{5y},u_{6x},u_{6y})
    &= 
    \sum_{j =1}^{6} f_j( \bm{u}_j(\bm{u}_1,\bm{u}_5,\bm{u}_6)).
\end{align}
To assess quasi-convexity of the reduced six-dimensional objective function, we perform a numerical verification based on the convexity of its sublevel sets.
Specifically, we consider the domain
\[
U = \left[-\frac{\theta}{b}, \frac{\theta}{b}\right]^6,
\]
which reflects the natural scaling of the variables. 
We first sample $20{,}000$ points uniformly from $U$ and evaluate the objective function $F$. Based on the empirical distribution of $F$ values, we select 
several representative levels $t$ corresponding to quantiles from $10\%$ to $90\%$ at increments of $10\%$. 
For each level $t$, we construct the approximate sublevel set
\[
Q_t = \{ x \in U : F(x) \le t \}.
\]
We then randomly sample pairs of points $(x,y)$ from $Q_t$ and evaluate multiple interior points on the line segment connecting them,
\[
z_\lambda = \lambda x + (1-\lambda) y, \quad \lambda \in (0,1).
\]
In our implementation, we use a uniform discretization of $(0,1)$ with nine interior points.
For each sampled pair and each $\lambda$, we verify whether
\[
F(z_\lambda) \le t.
\]
This provides a direct numerical test of the convexity of the sublevel set $Q_t$.
As shown in Table~\ref{table:subset}, no violations are observed across all tested levels, with a total of $30{,}000$ segment evaluations per level. Moreover, the quantity $F(\lambda x+(1-\lambda)y)-t$ remains strictly negative in all cases, with magnitude on the order of $10^{-4}$, which is consistent with sampling points close to the boundary of the sublevel sets.
These results provide strong numerical evidence supporting the convexity of the sublevel sets and hence the quasi-convexity of the objective function within the prescribed domain.
We have verified that Condition (S1) holds on $U$, which indicates that any stationary point is a local minimum. Combined with the quasi-convexity evidence above, this suggests that the minimizer is unique, if it exists.

\begin{table}[htbp]
\centering
\caption{Numerical verification of sublevel-set convexity in the 6D domain using quantiles from $10\%$ to $90\%$ with increments of $10\%$. For each level $t$, we report the number of sampled points in $Q_t$, the number of tested point pairs, the total number of segment evaluations, the number of violations of the condition $F(\lambda x+(1-\lambda)y)\le t$, and the maximum value of $F(\lambda x+(1-\lambda)y)-t$.}
\begin{tabular}{c|c|c|c|c|c}
\hline
Level $t$ & $\#Q_t$ & Pairs & Checks & Violations & Max violation \\
\hline
2.0067 & 2,000  & 3,000 & 30,000 & 0 & $-1.30\times 10^{-4}$ \\
2.1982 & 4,000  & 3,000 & 30,000 & 0 & $-2.65\times 10^{-5}$ \\
2.3429 & 6,000  & 3,000 & 30,000 & 0 & $-2.45\times 10^{-4}$ \\
2.4636 & 8,000  & 3,000 & 30,000 & 0 & $-1.15\times 10^{-6}$ \\
2.5803 & 10,000 & 3,000 & 30,000 & 0 & $-2.38\times 10^{-4}$ \\
2.6984 & 12,000 & 3,000 & 30,000 & 0 & $-1.03\times 10^{-5}$ \\
2.8178 & 14,000 & 3,000 & 30,000 & 0 & $-1.55\times 10^{-4}$ \\
2.9536 & 16,000 & 3,000 & 30,000 & 0 & $-5.34\times 10^{-6}$ \\
3.1318 & 18,000 & 3,000 & 30,000 & 0 & $-1.47\times 10^{-4}$ \\
\hline
\end{tabular}
\label{table:subset}
\end{table}

\begin{figure}[hbt]
	\centering
	\subfigure[]{
     \includegraphics[width=0.22\textwidth]{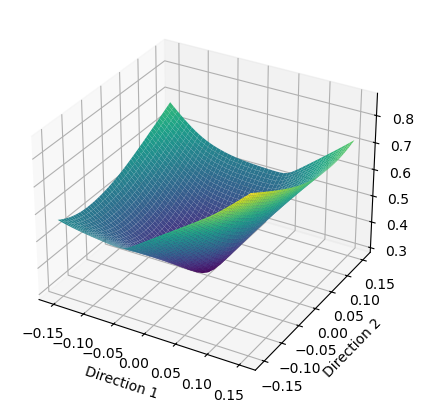} 
	}
    \subfigure[]{
    	\includegraphics[width=0.22\textwidth]{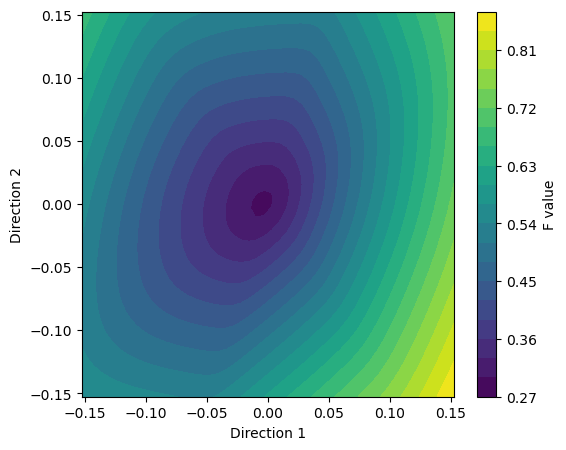} 
    }
    \subfigure[]{
     \includegraphics[width=0.22\textwidth]{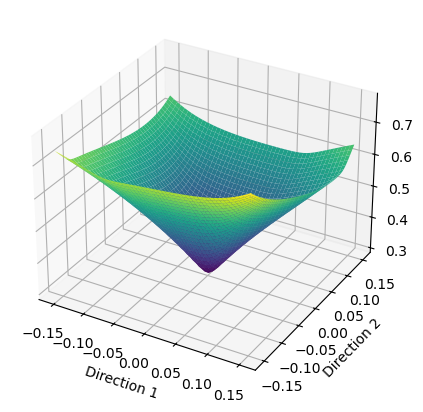} 
	}
    \subfigure[]{
    	\includegraphics[width=0.22\textwidth]{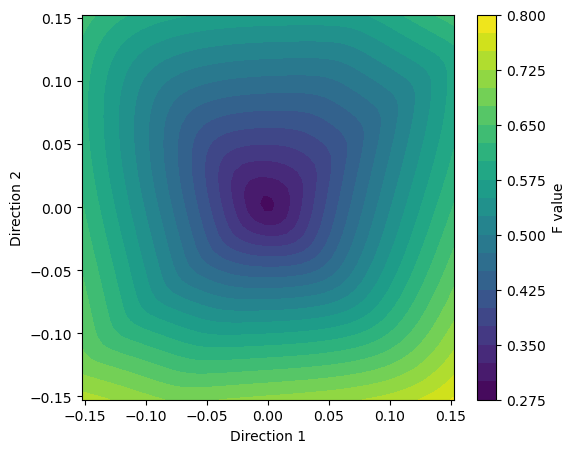} 
    }

    \subfigure[]{
     \includegraphics[width=0.22\textwidth]{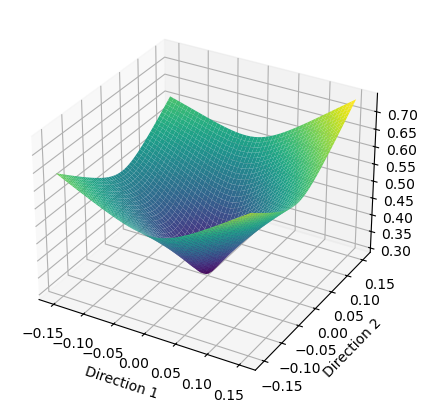} 
	}
    \subfigure[]{
    	\includegraphics[width=0.22\textwidth]{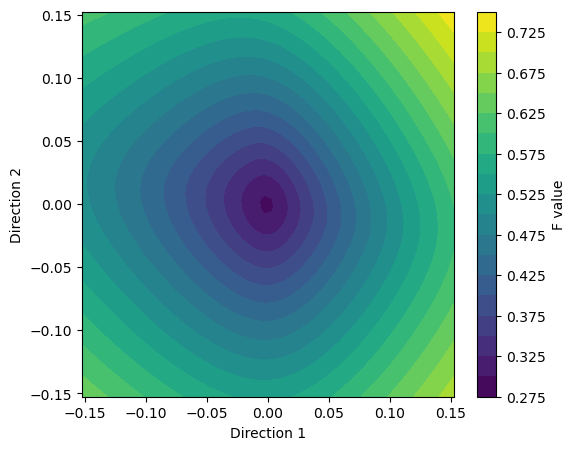} 
    }
    \subfigure[]{
     \includegraphics[width=0.223\textwidth]{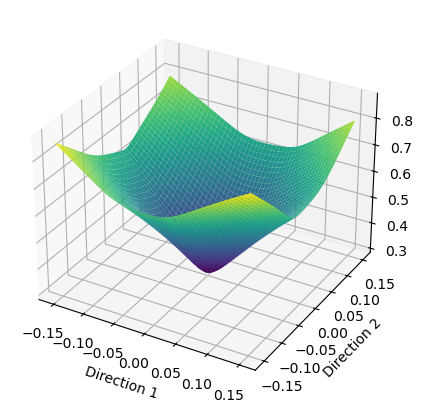} 
	}
    \subfigure[]{
    	\includegraphics[width=0.22\textwidth]{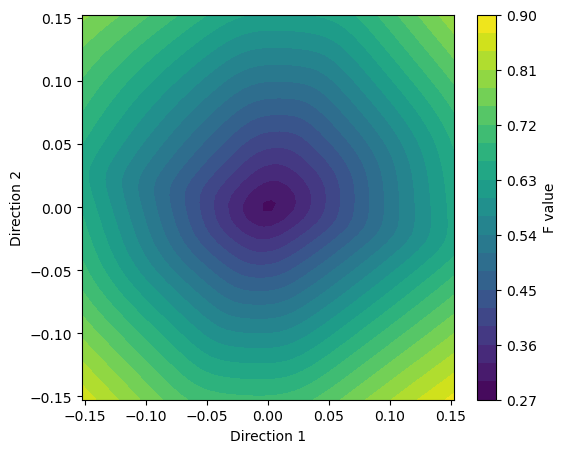} 
    }

    \subfigure[]{
     \includegraphics[width=0.22\textwidth]{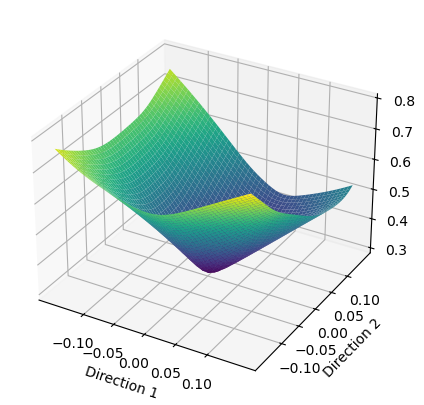} 
	}
    \subfigure[]{
    	\includegraphics[width=0.22\textwidth]{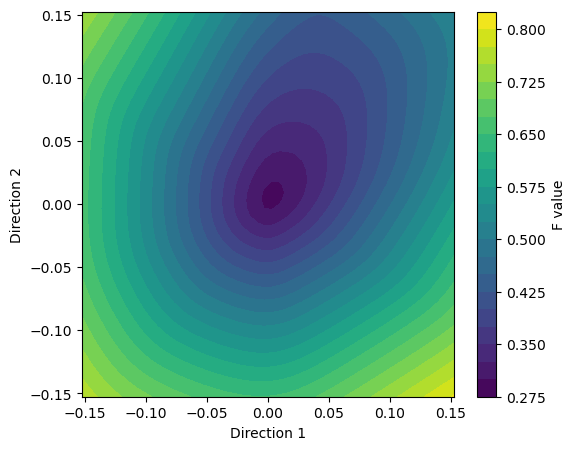} 
    }
    \subfigure[]{
     \includegraphics[width=0.22\textwidth]{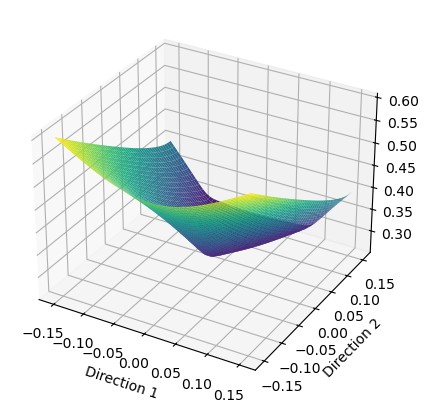} 
	}
    \subfigure[]{
    	\includegraphics[width=0.22\textwidth]{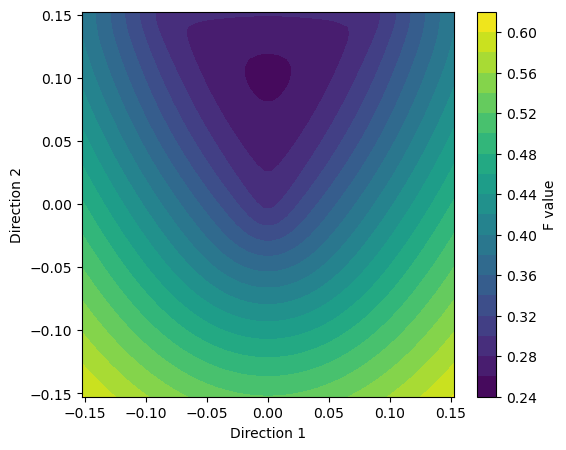} 
    }
	\caption{
    Visualization of the six dimensional function in Eq.\eqref{eqn:6d}:
    Graphs and contours of the objective function in some randomly selected planes spanned by orthonormal directions $v_1, v_2 \in \mathbb{R}^6$. 
    Here the vectors $v_1$ and $v_2$ in the panels (a)-(b), (c)-(d), (e)-(f), (g)-(h), (i)-(j), and (k)-(l) are: \\
    $(-0.255, -0.705, -0.272, -0.003, -0.572,  0.195), (-0.299, -0.307, -0.098,  0.549,  0.344, -0.622)$,
    $(-0.377, -0.297,  0.24,   0.659, -0.521,  0.084), (0.189,  0.133, -0.031,  0.711,  0.652, -0.123)$,
    $(-0.155,  0.06,   0.666, -0.725,  0.053, -0.026), (-0.335, -0.47,   0.522,  0.51,  -0.192, -0.312)$,
    $(-0.652, -0.041, -0.727,  0.105,  0.104, -0.153), (-0.102,  0.145,  0.286, -0.144,  0.75,  -0.551)$,
    $(0.886,  0.206,  0.297,  0.165,  0.213, -0.105), (-0.01,   0.341, -0.745,  0.382,  0.311, -0.294)$,
    $(1,0,0,0,0,0), (0,1,0,0,0,0)$, 
    respectively.
    }
    \label{fig}
\end{figure}

 In addition, we visualize this six-dimensional function $F(u_{1x},u_{1y},u_{5x},u_{5y},u_{6x},u_{6y})$  via randomly selected two-dimensional planes spanned by  orthonormal directions $v_1, v_2 \in \mathbb{R}^6$. 
We plot the objective function $f(\alpha_1, \alpha_2)=F(\alpha_1 v_1 + \alpha_2 v_2)$ with $\alpha_1, \alpha_2 \in [-\frac{\theta}{b}, \frac{\theta}{b}]$ and its contour lines on each plane; representative examples are shown in Figure~\ref{fig}.

\section{Conclusions}
\label{sec:conclusions}

We study a constrained nonconvex minimization model  proposed in Ref.~\cite{ZHANG} to compute the energy of the low angle grain
boundaries in crystalline materials. We propose a modified ADMM algorithm to solve the constrained
minimization problem and provide  convergence analysis. We notice that
the coefficient submatrix  
 corresponding to each subvariable block is a semiorthogonal
matrix multiplied by the length of Burgers vectors. Based on this
property, we prove that the modified ADMM with an increasing
penalty parameter is convergent without any strong
convexity assumptions on the objective function. Numerical examples show
that the modified ADMM converges much faster and gives more accurate results
compared to the penalty method and the augmented Lagrangian method. 
Moreover, we show that the minimum of this constrained minimization problem is unique when
the objective function satisfies quasi-convexity sufficient conditions over the
feasible region.
This modified ADMM with an increasing
penalty parameter can be applied to more general nonconvex constrained minimization problems, in
which the coefficient matrix in the constraint of each subvariable is of full column rank.

\begin{acknowledgements}
The work of LC  Zhang was supported by National Natural Science Foundation of China 12571464, 12201423, 12426311, Shenzhen Science \& Technology Program RCYX20231211090222026 and JCYJ20241202124209011, Research Team Cultivation Program of Shenzhen University 2023QNT011, and Huayuan Computing Technology Co., Ltd..
The work of Y Xiang was supported by Hong Kong Research Grants Council General Research Fund 16301720 and Collaborative Research Fund C1005-19G, and the Project of Hetao Shenzhen-HKUST Innovation Cooperation Zone HZQB-KCZYB-2020083. 
\end{acknowledgements}

\section*{Data availability}
The datasets generated in study are available upon reasonable request.

\section*{Competing Interests Declaration}
The authors have no relevant financial or non-financial interests to disclose.


\appendix
\section{Theorems for Section \ref{sec: unique stationary point} }

    \setcounter{theorem}{0}
    \renewcommand{\thetheorem}{\Alph{section}.\arabic{theorem}}

In this section of Appendix, we summarize some available theorems on the conditions for quasi-convexity.


\begin{theorem}[Sufficient condition of quasi-convexity \cite{CJ}]
	\label{th:sq}
	Let $h$ be twice differentiable function on an open convex set $U \subset \mathbb{R}^n$ such that
	
	(H1) $ \bm{x} \in U, \nabla h(\bm{x})=\bm{0} \Longrightarrow \bm{y}^\T \nabla^2 h(\bm{x}) \bm{y} > 0$  for every $\bm{y} \neq \bm{0},$
	
	(H2) $ \bm{x}\in U, \bm{y} \in \mathbb{R}^n,  \bm{y}^\T \nabla h(\bm{x})=0   \Longrightarrow \bm{y}^\T \nabla^2 h(\bm{x}) \bm{y} \geq 0$.
	
	Then $h(\bm{x})$ is quasi-convex on $U.$
\end{theorem}

Note that the condition (H2) is not easy to check. The following theorem gives  equivalent conditions.

\begin{theorem}[Equivalent second order conditions \cite{CrouzeixF82}]
	\label{th:equi}
	Let $ A$ be a real symmetric matrix of order $n,$
	and $\bm{a} \in \mathbb{R}^n, \bm{a} \neq \bm{0},$
	the real symmetric matrix of order $(n+1)$ is
	$$A_a = \left[\begin{array}{cc}
		0 &  \bm{a}^\T  \\
		\bm{a}  & A
	\end{array}\right].
	$$
	The following conditions are equivalent:
	
	(C1) $\bm{a}^\T y=0 $ implies
	$\bm{y}^\T A \bm{y} \geq 0.$
	
	(C2) Either $ A$ is positive semidefinite,
	or $A$ has one simple negative eigenvalue and there exists a vector $\bm{d} \in \mathbb{R}^n$ such that $A\bm{d}=\bm{a} $ and $\bm{a}^\T \bm{d} \leq 0.$
	
	(C3) The bordered hessian $A_a$ has one simple negative eigenvalue.
	
	(C4) For all nonempty subset $\mathcal{L} \subset \left\{1,2,\dots,n \right\},$
	$$\det D_{\mathcal{L}} =\det  \left[\begin{array}{cc}
		0 &  \bm{a}_{\mathcal{L}}  \\
		\bm{a}_{\mathcal{L}}  &  A_{\mathcal{L}}
	\end{array}\right]  \leq 0,$$
	where $A_{\mathcal{L}}$ is obtained from $A$ by deleting rows and columns whose indices are not in $\mathcal{L}, $ and $\bm{a}_{\mathcal{L}}$ is obtained analogously from $\bm{a}.$
\end{theorem}

%
%

\section{Theorems for Section \ref{sec: numerical} }\label{sec:convergences}

\setcounter{theorem}{0}

In this section of Appendix, we present the available
convergence theorems of the penalty method and the ALM, by which convergence can be shown when the two method applied to the constrained minimization problem  in Eq.~\eqref{eq:cp}.

\begin{theorem}[\cite{Edwin}]\label{th:p}
	Suppose that the objective function $f(x)$ is continuous and $\rho_k \rightarrow \infty$ and $k\rightarrow \infty. $ Then the limit of any convergent subsequence of the sequence $ \{x^{(k)} \}$ is a solution to the constrained optimization problem:
    \begin{equation}
        \begin{aligned}
            \min   &   f(x)    \\
           \text{s.t. } & h(x)=0,
        \end{aligned}
    \end{equation}
 where	$ x^{(k)}=\arg\min_{x} \{f(x)+ \rho_k \|h(x) \|^2\}.$
\end{theorem}

\begin{theorem}[\cite{Bertsekas}]\label{th:a}
Consider the constrained optimization problem
 \begin{equation}
        \begin{aligned}
            \min   &   f(x)    \\
           \text{s.t. } & h(x)=0.
        \end{aligned}
    \end{equation}
Assume that $f(x)$ and $h(x)$ are continuous functions, that $X $ is a closed set, and that the constraint set $\{ x \in X | h(x)=0 \}$ is nonempty. For $k =0,1,\dots$, let $x^{(k)} $ be a global minimum of the problem
       \begin{equation}
           \begin{aligned}
                \min  & L_{\rho_k}(x,\lambda^{(k)}) \\
                \text{s.t. } & x \in X,
           \end{aligned}
       \end{equation}
where $ L_{\rho_k}(x,\lambda^{(k)}) := f(x) + \lambda^{(k)}h(x) + \frac{\rho^{(k)}}{2} \|h(x)\|^2 $ is the augmented Lagrangian function,
and $ \{\lambda^{(k)} \}$ is bounded, $ 0< \rho_k < \rho_{k+1} $ for all $ k$, and $\rho_k \rightarrow \infty$. Then every limit point of the sequence $ \{ x^{(k)} \}$ is a global minimum of the original optimization problem.
 \end{theorem}

%

%
%

\bibliographystyle{spmpsci}      
\bibliography{references}   


\end{document}